\documentclass{article}
\usepackage{graphicx}
\usepackage[utf8]{inputenc}
\usepackage[T1]{fontenc}
\usepackage{subcaption}
\usepackage{amsmath,amssymb,lmodern}
\usepackage{amsmath}
\usepackage{amsthm}
\usepackage{caption}
\usepackage{color,soul}
\usepackage{float} 
\providecommand{\keywords}[1]{\textbf{Keywords} #1}
\newtheorem{remark}{Remark}
\newtheorem{result}{Result}
\newtheorem{theorem}{Theorem}

\begin{document}

\title{Subgrid multiscale stabilized finite element analysis of  non-Newtonian Casson model fully coupled with Advection-Diffusion-Reaction equations}

\author{ B.V. Rathish Kumar, Manisha Chowdhury\thanks{ Email addresses: drbvrk11@gmail.com (B.V.R. Kumar); chowdhurymanisha8@gmail.com(M.Chowdhury)} }
      
\date{Indian Institute of Technology Kanpur \\ Kanpur, Uttar Pradesh, India}

\maketitle

\begin{abstract}
In this paper we have studied subgrid multiscale stabilized formulation with dynamic subscales for non-Newtonian Casson fluid flow model tightly coupled with variable coefficients ADR ($VADR$) equation.  The Casson viscosity coefficient is taken to be dependent upon solute mass concentration. This paper presents the stability and convergence analyses of the stabilized finite element solution. The proposed expressions of the stabilization parameters helps in obtaining optimal order of convergences. Appropriate numerical experiments have been provided.
\end{abstract}


\section{Continuous problem}
\subsection{The model problem}
Let $\Omega \subset$ $\mathbb{R}^d$ ($d$=2,3) be an open bounded domain with the boundary $\partial \Omega$. Study of higher dimensional strongly coupled Casson fluid flow-transport model is a quite straight forward extension of its' lesser dimensional version, we have studied here two dimensional coupled system of equations for keeping simplicity in theoretical derivations specially. Now the general governing equations representing incompressible fluid flow phenomena is to find velocity $\textbf{u}=(u_1,u_2)$: $\Omega \times (0,T)$ $\rightarrow \mathbb{R}^2$ and pressure $p$: $\Omega \times (0,T)$ $\rightarrow \mathbb{R}$ such that
\begin{equation}
\begin{split}
\rho \frac{\partial u_1}{\partial t} + \rho (u_1 \frac{\partial u_1}{\partial x_1}+ u_2 \frac{\partial u_1}{\partial x_2}) & = - \frac{\partial p}{\partial x_1} + (\frac{\partial \tau_{11}}{\partial x_1}+ \frac{\partial \tau_{12}}{\partial x_2})+f_1^F \hspace{2mm} in \hspace{2mm} \Omega \times (0,T) \\
\rho \frac{\partial u_2}{\partial t} + \rho (u_1 \frac{\partial u_2}{\partial x}+ u_2 \frac{\partial u_2}{\partial x_2}) & = - \frac{\partial p}{\partial y} + (\frac{\partial \tau_{12}}{\partial x_1}+ \frac{\partial \tau_{22}}{\partial x_2})+f_2^F \hspace{2mm} in \hspace{2mm} \Omega \times (0,T) \\
\frac{\partial u_1}{\partial x_1} + \frac{\partial u_2}{\partial x_2} & = 0 \hspace{2mm} in \hspace{2mm} \Omega \times (0,T) \\
\end{split}
\end{equation}
where $\rho$ is density of the fluid, $f_1^F, f_2^F$ are the body forces and according to the rheological equation of state for an incompressible, isotropic Casson fluid \cite{29} the components of the shear stress tensor are assumed here as follows:
\begin{equation}
\tau_{ij}= 2 \mu(c,J_2) D_{ij}= \mu(c,J_2) (\frac{\partial u_i}{\partial x_j}+\frac{\partial u_j}{\partial x_i}) \hspace{2mm} for \hspace{2mm} i,j=1,2
\end{equation}
Here $D_{ij}$ denotes component of deformation-rate for each $i,j=1,2$ and 
\begin{equation}
\mu(c, J_2) = (\frac{\sqrt{\tau_y}}{\sqrt{2}} + \sqrt{\eta(c)} J_2^{\frac{1}{4}})^2 J_2^{-\frac{1}{2}}
\end{equation}
where $\tau_y$ denotes yield stress, $J_2= \{ 2 (\frac{\partial u_1}{\partial x_1})^2 + 2 (\frac{\partial u_2}{\partial x_2})^2 + (\frac{\partial u_1}{\partial x_2}+ \frac{\partial u_2}{\partial x_1})^2 \}$ and $\eta(c)$ is viscosity coefficient depending upon the concentration of the solute, transportation of which is modeled through the following $VADR$ equation.
\begin{equation}
\frac{\partial c}{\partial t}- \bigtriangledown \cdot \tilde{\bigtriangledown} c + \textbf{u} \cdot \bigtriangledown c + \alpha c  = f^T \hspace{2mm} in \hspace{2mm} \Omega \times (0,T) \\
\end{equation}
where the notation $\tilde{\bigtriangledown}: = (D_1 \frac{\partial}{\partial x}, D_2 \frac{\partial}{\partial y})$ \\
$D_1,D_2$ are the variable coefficients of diffusion, $\alpha$ is reaction coefficient and $f^T$ is the source of solute mass. \vspace{1mm}\\
Let us consider homogeneous Dirichlet boundary conditions for both (1) and (4) and the initial conditions are assumed respectively as follows:
\begin{equation}
\begin{split}
\textbf{u}= \textbf{0}, c & =0 \hspace{2mm} on \hspace{2mm} \partial \Omega \times (0,T)\\
\textbf{u}= \textbf{u}_0, c & = c_0 \hspace{2mm} at \hspace{2mm} t=0
\end{split}
\end{equation}
This coupled system can be expressed in the following operator form
\begin{equation}
M \partial_t \textbf{U} + \mathcal{L}(\textbf{u}, \mu(c,J_2); \textbf{U})= \textbf{F}
\end{equation}
where $\textbf{U}$ denotes the triplet $(\textbf{u},p,c)$ and $M$, a matrix= $diag(\rho,\rho,0,1)$, $\partial_t \textbf{U}=[\frac{\partial \textbf{u}}{\partial t},\frac{\partial p}{\partial t} \frac{\partial c}{\partial t}]^T$ and $\textbf{F}=[f_1^F,f_2^F,0,f^T]^T$\\

\[
\mathcal{L} (\textbf{u},\mu; \textbf{U})=
  \begin{bmatrix}
  \rho(u_1 \frac{\partial u_1}{\partial x}+ u_2 \frac{\partial u_1}{\partial y}) + \frac{\partial p}{\partial x} - \{ \frac{\partial}{\partial x} (2 \mu \frac{\partial u_1}{\partial x})+  \frac{\partial}{\partial y} ( \mu (\frac{\partial u_2}{\partial x} + \frac{\partial u_1}{\partial y})\} \\
  \rho(u_1 \frac{\partial u_2}{\partial x}+ u_2 \frac{\partial u_2}{\partial y}) + \frac{\partial p}{\partial y} - \{ \frac{\partial}{\partial x} ( \mu (\frac{\partial u_2}{\partial x} + \frac{\partial u_1}{\partial y})+  \frac{\partial}{\partial y} (2 \mu \frac{\partial u_2}{\partial y})\} \\
    \bigtriangledown \cdot \textbf{u} \\
    - \bigtriangledown \cdot \tilde{\bigtriangledown} c + \textbf{u} \cdot \bigtriangledown c + \alpha c 
  \end{bmatrix}
\]
where the notation $\mu$ denotes abbreviated form of  non-linear $\mu(c,J_2)$. Now let us make few suitable assumptions on coefficients of the coupled system in the following:\vspace{1mm}\\
\textbf{(i)} The fluid viscosity $\eta(c) \in C^0(\mathbb{R}^+; \mathbb{R}^+)$, the space of positive real valued functions defined on positive real numbers and hence there exists $\eta_l$ and $\eta_u$ such that 
\begin{equation}
0 < \eta_l \leq \eta(x) \leq \eta_u \hspace{2mm} for \hspace{2mm} any \hspace{2mm} x\in \mathbb{R}^+
\end{equation}
\textbf{(ii)} $D_1= D_1((x,y),t) \in C^0(\mathbb{R}^2 \times [0,T];\mathbb{R})$ and $D_2= D_2((x,y),t) \in C^0(\mathbb{R}^2 \times [0,T];\mathbb{R})$ where $ C^0(\mathbb{R}^2\times [0,T];\mathbb{R})$ is the space of real valued continuous function defined on $\mathbb{R}^2$ for fixed $t \in [0,T]$. \vspace{2mm} \\
\textbf{(iii)} $\rho$, $\tau_y$ and $\alpha$ are positive constants. \vspace{2mm}\\
\textbf{(iv)}  The spaces of continuous solution $(\textbf{u},p,c)$ are assumed as: \\ $\textbf{u} \in L^{\infty}(0,T; (H^2(\Omega))^2)\bigcap C^{0}(0,T; (H_0^1(\Omega))^2)$ and $\textbf{u}_{tt} \in C^{0}(0,T; (H^1(\Omega))^2)$ \\
$p \in L^{\infty}(0,T;H^1(\Omega))\bigcap C^{0}(0,T;L^2_0(\Omega)) $,
 $c \in L^{\infty}(0,T;H^2(\Omega))\bigcap C^0(0,T;H^1_0(\Omega))$ \\
 and $c_{tt} \in C^{0}(0,T; H^1(\Omega))$\vspace{1mm}\\
\textbf{(v)} Additional assumptions on exact velocity solution are choosing $u_i$ for $i=1,2$ such that  all the combinations of $(\frac{\partial u_i}{\partial x_j} + \frac{\partial u_j}{\partial x_i})^2$  for $i,j=1,2$ are  bounded functions on $\Omega$ for each $t \in (0,T)$.

\subsection{Weak formulation}
Assuming the body forces and the source term $f_1^F,f_2^F,f^T \in L^2(0,T; L^2(\Omega))$ the appropriate spaces to derive the weak formulation are $V=H^1_0(\Omega) $ and $Q=L^2_0(\Omega)$. Now denoting the product space $V \times V \times Q \times V$ by $ \bar{\textbf{V}}$ the variational formulation of (6) is to find  \textbf{U}(t)= (\textbf{u}(t),p(t),c(t)) $ \in \bar{\textbf{V}}$ such that $\forall$ \textbf{V}=(\textbf{v},q,d) $\in  \bar{\textbf{V}}$ and for $a.e.$ $t \in J$
\begin{equation}
\begin{split}
(M\partial_t \textbf{U},\textbf{V}) + B(\textbf{u},\mu; \textbf{U}, \textbf{V}) & = L(\textbf{V})   \hspace{2 mm} \forall \textbf{V} \in \bar{\textbf{V}}
\end{split}
\end{equation} 
where $(M \partial_t \textbf{U}, \textbf{V})= \rho \int_{\Omega} \frac{\partial u_1}{\partial t} v_1 + \rho \int_{\Omega} \frac{\partial u_2}{\partial t} v_2 +\int_{\Omega} \frac{\partial c}{\partial t} d$ \vspace{1mm}\\
$B(\textbf{u},\mu; \textbf{U}, \textbf{V})$= $ a^F_{TL}(\textbf{u},\textbf{u},\textbf{v}) + a_{NL}^F(\mu; \textbf{u},\textbf{v})- b(\textbf{v},p)
+ b(\textbf{u},q)+ a_{L}^T(c,d)+ a_{TL}^T(\textbf{u},c,d) $ and  the linear functional $L(\textbf{V})= l^{F}(\textbf{v})+ l^{T}(d)$.\vspace{1mm}\\
where the notations are defined in the following:\\
$a^F_{TL}(\textbf{u},\textbf{v},\textbf{w})=\rho \int_{\Omega} ((\textbf{u} \cdot \bigtriangledown)\textbf{v})\cdot \textbf{w}+ \frac{\rho}{2} \int_{\Omega} (\bigtriangledown \cdot \textbf{u})\textbf{v} \cdot \textbf{w} $ and $b(\textbf{v},q)= \int_{\Omega} (\bigtriangledown \cdot \textbf{v}) q$\vspace{1 mm}\\
$a_{NL}^F(\mu; \textbf{u},\textbf{v})$= $\int_{\Omega} \mu(c,J_2) \{2 \frac{\partial u_1}{\partial x} \frac{\partial v_1}{\partial x} + \frac{\partial u_2}{\partial x} \frac{\partial v_1}{\partial y}+ \frac{\partial u_1}{\partial y} \frac{\partial v_2}{\partial x}+ \frac{\partial u_1}{\partial y} \frac{\partial v_1}{\partial y}+ \frac{\partial u_2}{\partial x} \frac{\partial v_2}{\partial x}+ 2 \frac{\partial u_2}{\partial y} \frac{\partial v_2}{\partial y}\}$;
 $a_{L}^T(c,d) = \int_{\Omega} \tilde{\bigtriangledown}c \cdot \bigtriangledown d  + \alpha\int_{\Omega}cd $ and
 $a_{TL}^T(\textbf{u},c,d)= \int_{\Omega} d \textbf{u} \cdot \bigtriangledown c$\vspace{1 mm} \\
 $l^{F} (\textbf{v})= \int_{\Omega} (f_1^F v_1+ f_2^F v_2)$ and  $l^{T}(d)= \int_{\Omega} f^Td$ \vspace{1 mm} \\
The addition of the incompressibility condition into the trilinear term $a^F_{TL}(\cdot,\cdot,\cdot)$ makes it equivalent to it's original form obtained from non-linear convective term in (1). This modification provides the following important properties of $a^F_{TL}(\cdot,\cdot,\cdot)$: \vspace{1mm}\\
(\textbf{a}) for each $\textbf{u} \in V \times V$, $a^F_{TL}(\textbf{u},\textbf{v},\textbf{v})=0$ \hspace{1mm} $\forall$ $\textbf{v} \in V \times V$ \vspace{1mm}\\
(\textbf{b}) for \textbf{u}, \textbf{v}, \textbf{w} $\in V \times V$
\begin{equation}
  a^F_{TL}(\textbf{u},\textbf{v},\textbf{w})\leq \begin{cases}
    C \|\textbf{u}\|_1 \|\textbf{v}\|_1 \|\textbf{w}\|_1 & \\
    C \|\textbf{u}\|_0 \|\textbf{v}\|_2 \|\textbf{w}\|_1 & \\
    C \|\textbf{u}\|_2 \|\textbf{v}\|_1 \|\textbf{w}\|_0 & \\
    C \|\textbf{u}\|_0 \|\textbf{v}\|_1 \|\textbf{w}\|_{L^{\infty}(\Omega)} 
  \end{cases}
\end{equation}
where $C$ is a constant and $\| \cdot \|_i$ for i=0,1,2 denote the standard $L^2, H^1, H^2$ full norms respectively.  From now onward for simplicity we use $\| \cdot \|$ instead of $\| \cdot \|_0$ to denote $L^2(\Omega)$ norm. The bilinear form $b(\textbf{v},q)$ also satisfies the continuous $inf$-$sup$ condition for this pair of velocity pressure spaces \cite{30}.

\section{Finite element formulations}
\subsection{Discretizations of space and time}
Let the domain $\Omega$ be discretized into $n_{el}$ numbers of subdomains $\Omega_k$ for $k$=1,2,..., $n_{el}$. Let $h_k$ be the diameter of each subdomain $\Omega_k$ where $h$ denotes $\underset{k=1,...n_{el}}{max} h_k$.
Let the  finite dimensional spaces  $V_h \subset V$ and $Q_h \subset Q$ be considered as: \vspace{1mm}\\
$V_h= \{ v \in V: v(\Omega_k)= \mathcal{P}^l(\Omega_k)\} $ and 
$Q_h= \{ q \in Q : q(\Omega_k)= \mathcal{P}^{l-1}(\Omega_k)\}$ \vspace{1 mm}\\
where  $\mathcal{P}^l(\Omega_k)$ denotes complete polynomial upto order $l$ over each $\Omega_k$ for $k$=1,2,..., $n_{el}$ along with this inclusion assumption $\bigtriangledown \cdot V_h \subset Q_h$. Considering a likewise notation $\bar{\textbf{V}}_h$ for denoting the product space $V_h \times V_h \times Q_h \times V_h$, the standard \textbf{Galerkin finite element formulation} for the variational form (8) is to find $\textbf{U}_h(t) $= $(\textbf{u}_h(t),p_h(t),c_h(t))$ $ \in \bar{\textbf{V}}_h$ such that $\forall$ $\textbf{V}_h=(\textbf{v}_h,q_h,d_h)$ $\in \bar{\textbf{V}}_h$ and for $a.e.$ $t \in J$
\begin{equation}
(M\partial_t \textbf{U}_h,\textbf{V}_h) + B(\textbf{u}_h, \mu_h; \textbf{U}_h, \textbf{V}_h) = L(\textbf{V}_h)   
\end{equation}
where $(M\partial_t \textbf{U}_h,\textbf{V}_h)$= $ \rho (\frac{\partial u_{1h}}{\partial t}, v_{1h})+ \rho (\frac{\partial u_{2h}}{\partial t}, v_{2h})+ (\frac{\partial c_h}{\partial t}, d_h)$ \vspace{1mm}\\
$B(\textbf{u}_h,\mu_h; \textbf{U}_h, \textbf{V}_h)$= $a^F_{TL}(\textbf{u}_h,\textbf{u}_h,\textbf{v}_h) + a_{NL}^F(\mu_h; \textbf{u}_h,\textbf{v}_h)- b(\textbf{v}_h,p_h)
+ b(\textbf{u}_h,q_h)+ a_{L}^T(c_h,d_h)+ a_{TL}^T(\textbf{u}_h,c_h,d_h)$\\
 and  the linear functional $L(\textbf{V}_h)= l^{F}(\textbf{v}_h)+ l^{T}(d_h)$.\vspace{1mm}\\
 Here $\mu_h$ denotes the discrete expression of $\mu$ in (3) after imposing it in finite dimensional setting.
In addition let us consider the initial conditions $(\textbf{u}_h, \textbf{v}_h)\mid_{t=0}=(\textbf{u}_0, \textbf{v}_h)$ $\forall \textbf{v}_h \in V_h \times V_h $ and
  $(c_h,d_h)\mid_{t=0}= (c_0,d_h)$ $\forall d_h \in V_h$. The finite element solutions also satisfy the following inverse inequalities for regular partition \cite{31} for all $v_h \in V_h$
  \begin{center}
  $\|\Delta v_h\| \leq C_I h^{-1} \|\nabla v_h\|$ \hspace{1mm}  and  \hspace{1mm}   $\|\nabla v_h\| \leq C_I h^{-1} \|v_h\|$
  \end{center}
The $backward$ $Euler$ fully implicit time discretization scheme here begins with dividing $T$ into $N$ number of small time steps and the $n^{th}$ time step is defined by $t_n$= $ndt$ where the time step $dt=\frac{T}{N}$. 
Let $\textbf{u}^{n+1},p^{n+1},c^{n+1}$ be approximations of $\textbf{u}(\textbf{x},t^{n+1}), p(\textbf{x},t^{n+1}),c(\textbf{x},t^{n+1})$ respectively. Now by Taylor series expansion \cite{32},we have 
\begin{equation}
\begin{split}
\frac{ \textbf{u}^{n+1}-\textbf{u}^n}{dt} & = \frac{\partial \textbf{u}}{\partial t}(\textbf{x},t^{n+1}) + TE_{\textbf{u}}\mid_{t=t^{n+1}} \\
\frac{c^{n+1}-c^n}{dt} & = \frac{\partial c}{\partial t}(\textbf{x},t^{n+1}) + TE_c \mid_{t=t^{n+1}} 
\end{split}
\end{equation}
where the truncation errors together denoted by  $\textbf{TE} \mid_{t=t^{n+1}}$ $\simeq$ $\textbf{TE}^{n+1}$,  depend upon the time-derivatives of the respective variables and the time step $dt$ in the following way \cite{34} and later applying assumption \textbf{(iv)} on the time derivative terms:
\begin{equation}
\begin{split}
\| TE_{ \textbf{u}}^{n+1}\| & \leq C'_1 dt \| \textbf{u}_{tt}^{n+1}\|_{L^{\infty}(t^n,t^{n+1},L^2)} \leq \bar{C}_1 dt \\
\| TE_{c}^{n+1}\| & \leq C'_2 dt \| c_{tt}^{n+1}\|_{L^{\infty}(t^n,t^{n+1},L^2)} \leq \bar{C}_2 dt \\
\end{split}
\end{equation}
After introducing all the required definitions the fully discrete Galerkin finite element formulation is to find $\textbf{U}_h^{n+1}= (\textbf{u}_h^{n+1},p_h^{n+1},c_h^{n+1}) \in \bar{\textbf{V}}_h$ for given $\textbf{U}_h^n = (\textbf{u}_h^n,p_h^n,c_h^n)\in \bar{\textbf{V}}_h$ such that $\forall \hspace{1mm} \textbf{V}_h=(\textbf{v}_h,q_h,d_h) \in \bar{\textbf{V}}_h $
\begin{equation}
(M\frac{(\textbf{U}_h^{n+1}-\textbf{U}_h^n)}{dt}, \textbf{V}_h)+ B(\textbf{u}_h^{n}, \mu^n_h ;\textbf{U}_h^{n+1}, \textbf{V}_h) = L(\textbf{V}_h) + (\textbf{TE}^{n+1},\textbf{V}_h)  
\end{equation}
\begin{remark}
Galerkin formulation admits numerical instabilities for convection dominated flows and for few pairs of velocity-pressure spaces failing to satisfy $inf$-$sup$ condition. Though the pair of spaces chosen here satisfies the discrete $inf$-$sup$ condition \cite{33}. The stabilized methods circumvent this compatibility condition and consequently any pair of finite dimensional spaces can be considered.
\end{remark}

\subsection{ Subgrid multiscale stabilized formulation}
In order to overcome the instabilities in the standard finite element formulation the derivation of this stabilized method begins with additive decomposition of the continuous solution $\textbf{U}$ into coarse scales and fine scales solutions. Here the computed solution $\textbf{U}_h \in \bar{\textbf{V}}_h$ is considered to be the coarse scales solution and the unknown fine scales solution $\tilde{\textbf{U}} \in \textbf{V}'$, also known as $subgrid$ scales be derived analytically, where $\textbf{V}'$ a subspace of $\bar{\textbf{V}}$ completes $\bar{\textbf{V}}_h$ in $\bar{\textbf{V}}$.  Applying similar decomposition on the test function $\textbf{V}=\textbf{V}_h+ \tilde{\textbf{V}}$, the weak formulation (8) can be equivalently expressed as
\begin{equation}
\begin{split}
(M \partial_t \textbf{U}_h+M \partial_t \tilde{\textbf{U}},\textbf{V}_h)+B(\textbf{u}_h, \mu_h; \textbf{U}_h+ \tilde{\textbf{U}}, \textbf{V}_h) & = L(\textbf{V}_h) \\
(M \partial_t \textbf{U}_h+M \partial_t \tilde{\textbf{U}},\tilde{\textbf{V}})+B(\textbf{u}_h, \mu_h; \textbf{U}_h +  \tilde{\textbf{U}}, \tilde{\textbf{V}}) & = L(\tilde{\textbf{U}}) \\
\end{split}
\end{equation}
Now our aim is to find an analytical expression for the unknown subgrid scale $\tilde{\textbf{U}}$ in terms of $\textbf{U}_h$. In this regard we perform element wise integration on the second sub-equation of (14) and approximate the differential operator $\mathcal{L}$ by an algebraic operator $\tau_k$ over each sub domain $\Omega_k$ we have
\begin{equation}
M \partial_t \tilde{\textbf{U}} + \tau_k^{-1} \tilde{\textbf{U}} =\textbf{R}_h:= \textbf{F}-M \partial_t \textbf{U}_h- \mathcal{L}(\textbf{u}_h, \mu_h; \textbf{U}_h) \hspace{1mm} on \hspace{1mm} \Omega_k
\end{equation}
where $\textbf{R}_h$ is the residual vector. Each of the stabilization parameters \cite{3} is of the following form:
\begin{equation}
\begin{split}
\tau_{1k} &= \tau_{1}= (c_1 \frac{\eta_0}{h^2}+  c_2 \frac{\rho \mid \textbf{u}_h \mid}{h})^{-1} \\
\tau_{2k} &=\tau_{2}=\frac{h^2}{c_1 \tau_{1}} \\
\tau_{3k} & = \tau_{3}= c_3(\frac{9D_m}{4h^2} + \frac{3\mid \textbf{u}_h \mid}{2h} + \alpha )^{-1}
\end{split}
\end{equation}
where $c_1,c_2,c_3$ are suitably chosen positive parameters independent of viscosity, $\eta_0$ is total viscosity and $\mid \textbf{u}_h \mid $ of the computed velocity and $D_m$ is deduced in \cite{10}. 
Now applying the $backward$ $Euler$ time discretization scheme on the above equation (15) we have
\begin{equation}
\tilde{\textbf{U}}= \bar{\tau}_k(\textbf{R}^h+ \bar{\textbf{d}})
\end{equation}
where the stabilization parameter $\tau_k$ = $diag(\tau_{1k},\tau_{1k},\tau_{2k}, \tau_{3k})$ and $\bar{\tau}_k$= $(\frac{1}{dt}M+ \tau_k^{-1})^{-1}$= $diag(\frac{\tau_{1k} dt}{dt+ \rho \tau_{1k}}, \frac{\tau_{1k} dt}{dt+ \rho \tau_{1k}}, \tau_{2k}, \frac{\tau_{3k} dt}{dt+  \tau_{3k}})= diag(\bar{\tau}_{1k}, \bar{\tau}_{1k}, \bar{\tau}_{2k}, \bar{\tau}_{3k})$ (say) and $\bar{\textbf{d}}$= $\sum_{i=0}^{N}(\frac{1}{dt}M \bar{\tau}_k)^i(\textbf{F} -M\partial_t \textbf{U}_h - \mathcal{L}(\textbf{u}_h, \mu_h ;\textbf{U}_h))$.\vspace{1mm}\\
Now we perform element wise integration on that part of $B(\cdot,\cdot; \cdot, \cdot)$ in the first sub-equation of (14) which contains $\tilde{\textbf{U}}$ to separate it out in the following way
\begin{equation}
(M \partial_t \textbf{U}_h + M \partial_t \tilde{\textbf{U}},\textbf{V}_h)+B(\textbf{u}_h, \mu_h; \textbf{U}_h, \textbf{V}_h) + \sum_{k=1}^{n_{el}} (\tilde{\textbf{U}}, \mathcal{L}^*(\textbf{u}_h, \mu_h; \textbf{V}_h) )_k   = L(\textbf{V}_h)
\end{equation}
where $(\cdot,\cdot)_k$ denotes inner product on an element $\Omega_k$. Now consecutive substitution of the results (15)-(16) into the above equation (17) results in the following \textbf{ algebraic subgrid multiscale} ($ASGS$) stabilized finite element formulation: to 
find $\textbf{U}_h(t) $= $(\textbf{u}_h(t),p_h(t),c_h(t))$ $ \in \bar{\textbf{V}}_h$ such that $\forall$ $\textbf{V}_h=(\textbf{v}_h,q_h,d_h)$ $\in \bar{\textbf{V}}_h$ and for $a.e.$ $t \in J$
\begin{equation}
(M\partial_t \textbf{U}_h,\textbf{V}_h) + B_{S}(\textbf{u}_h, \mu_h ; \textbf{U}_h, \textbf{V}_h)  = L_{S}(\textbf{V}_h)  
\end{equation}
where $B_{S}(\textbf{u}_h,  \mu_h ;\textbf{U}_h, \textbf{V}_h)= B(\textbf{u}_h,  \mu_h; \textbf{U}_h, \textbf{V}_h)+ \sum_{k=1}^{n_{el}} (\tau_k' (M\partial_t \textbf{U}_h+ \mathcal{L}(\textbf{u}_h, \mu_h ;\textbf{U}_h)-\textbf{d}), -\mathcal{L}^*(\textbf{u}_h, \mu_h;\textbf{V}_h))_{\Omega_k}- \sum_{k=1}^{n_{el}}((I-\tau_k^{-1}\tau_k') (M\partial_t \textbf{U}_h + \mathcal{L}(\textbf{u}_h, \mu_h;\textbf{U}_h)), \textbf{V}_h)_{\Omega_k} 
-\sum_{k=1}^{n_{el}} (\tau_k^{-1}\tau_k' \textbf{d}, \textbf{V}_h)_{\Omega_k}$ \vspace{1 mm}\\
$L_{S}(\textbf{V}_h)= L(\textbf{V}_h)+ \sum_{k=1}^{n_{el}}(\tau_k' \textbf{F}, -\mathcal{L}^*(\textbf{u}_h, \mu_h;\textbf{V}_h))_{\Omega_k}- \sum_{k=1}^{n_{el}}((I-\tau_k^{-1}\tau_k')\textbf{F}, \textbf{V}_h)_{\Omega_k}$
\begin{remark}
One important aspect of the method employed here is both the subscales are time dependent and hence the method is known as subgrid multiscale method with dynamic subscales. This feature makes this method more accurate in comparison  with subgrid method with quasi static subscales, where the unknown subgrid scales is considered time independent. Here we have not studied convection tracking as the effect of subgrid scale on convective term is insignificant in the study of convergence analysis.
\end{remark}
\subsubsection{Stability analysis}
Stability of the pressure term is ensured by holding of $inf$-$sup$ condition for these choices of spaces. In this section we analyze stability of the solutions with respect to the given data. \\
On applying $backward$ $Euler$ time discretization scheme the stabilized formulation (19) can be equivalently expressed as the following system of equations: for given $\textbf{U}_h^{n}= (\textbf{u}_h^{n},p_h^{n},c_h^{n}) \in \bar{\textbf{V}}_h$ and $\tilde{\textbf{U}}^{n}=(\tilde{\textbf{u}}^n, \tilde{p}^n,\tilde{c}^n) \in \textbf{V}'$ find $\textbf{U}_h^{n+1}= (\textbf{u}_h^{n+1},p_h^{n+1},c_h^{n+1}) \in \bar{\textbf{V}}_h$ and $\tilde{\textbf{U}}^{n+1}=(\tilde{\textbf{u}}^{n+1}, \tilde{p}^{n+1},\tilde{c}^{n+1}) \in \textbf{V}'$ such that $\forall \hspace{1mm} \textbf{V}_h \in \bar{\textbf{V}}_h $
\begin{multline}
(M \frac{\textbf{U}_h^{n+1}-\textbf{U}_h^n}{dt}, \textbf{V}_h)+B(\textbf{u}_h^n,\mu_h^n; \textbf{U}_h^{n+1}, \textbf{V}_h)+ (\tilde{\textbf{U}}^{n+1},\mathcal{L}(\textbf{u}_h,\mu_h;\textbf{V}_h)) \\
- (\tilde{\textbf{U}}^{n+1}, M \partial_t \textbf{V}_h) = L(\textbf{V}_h)\\
 M \frac{\tilde{\textbf{U}}^{n+1}-\tilde{\textbf{U}}^n}{dt} + \tau_k^{-1}\tilde{\textbf{U}}^{n+1}  = \textbf{F}^{n+1}- M \partial_t \textbf{U}_h^{n+1}- \mathcal{L}(\textbf{u}_h^n,\mu_h^n;\textbf{U}_h^{n+1})
\end{multline}
The assumptions $\textbf{u}_0,c_0 \in L^2(\Omega)$ imply uniform boundedness of $\|\textbf{u}_h^0\|$ and $\|c_h^0\|$ and we assume the initial condition for subgrid scales $\tilde{\textbf{U}}|_{t=0}=0$. Again the assumptions made in section 2.2 on the admissible spaces of the body forces and source term imply the corresponding fully-discrete terms $\{f_1^{F^n}\}, \{f_2^{F^n}\}, \{f^{T^n}\} \in l^2(0,T;L^2(\Omega))$.
\begin{theorem}
For $(\textbf{u}_h^{n+1},p_h^{n+1},c_h^{n+1})$ and $(\tilde{\textbf{u}}^{n+1}, \tilde{p}^{n+1},\tilde{c}^{n+1})$ being solutions of (20) and for the finite element solutions satisfying the inverse inequalities for regular partitions, the following stability bounds hold for all $dt>0$
\begin{multline}
\underset{n=0,...,N-1}{max} \{ \|\textbf{u}^{n+1}_h\|^2 +\|c^{n+1}_h\|^2+ \| \tilde{\textbf{u}}^{n+1}\|^2+\|\tilde{c}^{n+1}\|^2 \} +\\
 \sum_{n=0}^{N-1}dt (C_1^s \|\nabla \textbf{u}_h^{n+1}\|^2 + C_2^s  \|\textbf{u}_h^{n+1}\|^2 + C_3^s \| \tau_{1k}^{-\frac{1}{2}} \tilde{\textbf{u}}^{n+1}\|^2)+ \\
 \sum_{n=0}^{N-1}dt (C_4^s \|\nabla c_h^{n+1}\|^2 + C_5^s  \|c_h^{n+1}\|^2 + C_6^s \| \tau_{3k}^{-\frac{1}{2}} \tilde{c}^{n+1}\|^2)\\
  \leq C_7^s \{ \sum_{n=0}^{N-1} dt( \|f_1^{F^{n+1}}\|^2 +\|f_2^{F^{n+1}}\|^2+\|f^{T^{n+1}}\|^2)+  \|\textbf{u}_h^{0}\|^2+\|c_h^0\|^2\}
\end{multline}
where $C_i^s$ are positive constants for i=1,...,7. Furthermore $\{f_1^{F^n}\}, \{f_2^{F^n}\}, \{f^{T^n}\} \in l^2(0,T;L^2(\Omega))$ and  uniformly bounded $\|\textbf{u}_h^0\|$ and $\|c_h^0\|$ implies that
\begin{center}
$ \{\tilde{\textbf{u} }^n\}, \{\tilde{c}^n\} \in l^{\infty}(0,T; L^2(\Omega)), \{ \tau_{1k}^{-\frac{1}{2}} \tilde{\textbf{u}}^{n+1} \}, \{ \tau_{3k}^{-\frac{1}{2}} \tilde{c}^{n+1} \} \in  l^2(0,T: L^2(\Omega))$ \vspace{1mm}\\
$\{\textbf{u}^n_h\},  \{c^n_h\} \in l^{\infty}(0,T; L^2(\Omega)) \bigcap l^2(0,T; H^1(\Omega)) $
\end{center}
\end{theorem}
\begin{proof}
Substituting $\textbf{V}_h$ by $(\textbf{u}_h^{n+1},p_h^{n+1},c_h^{n+1})$ in first sub equation of (20) and integrating the second one after multiplying it by $\tilde{\textbf{U}}^{n+1}$ on both sides we have
\begin{multline}
\frac{M}{dt}( \|\textbf{U}_h^{n+1}\|^2+ \|\tilde{\textbf{U}}^{n+1}\|^2)-(\frac{M}{dt} (\textbf{U}^n_h +  \tilde{\textbf{U}}^n), \textbf{U}_h^{n+1})+ \tau_k^{-1} \|\tilde{\textbf{U}}^{n+1}\|^2+2 \alpha \|c^{n+1}\|^2 +\\
B(\textbf{u}_h^n,\mu_h^n; \textbf{U}_h^{n+1}, \textbf{U}_h^{n+1})-2 \sum_{k=1}^{n_{el}}(\tilde{u}_1^{n+1}, \frac{\partial}{\partial x}(2 \mu_h \frac{\partial u_{1h}^{n+1}}{\partial x})+\frac{\partial}{\partial y}(\mu_h( \frac{\partial u_{2h}^{n+1}}{\partial x}+ \frac{\partial u_{1h}^{n+1}}{\partial y})))_k-\\
2 \sum_{k=1}^{n_{el}}(\tilde{u}_2^{n+1},\frac{\partial}{\partial x}(\mu_h( \frac{\partial u_{2h}^{n+1}}{\partial x}+ \frac{\partial u_{1h}^{n+1}}{\partial y}))+\frac{\partial}{\partial y}(2 \mu_h \frac{\partial u_{2h}^{n+1}}{\partial y}))_k-2 \sum_{k=1}^{n_{el}}(\tilde{c}^{n+1}, \nabla \cdot \tilde{\nabla} c_h^{n+1})_k\\
=(\textbf{F}^{n+1}, \textbf{U}_h^{n+1})+ (\textbf{F}^{n+1}, \tilde{\textbf{U}}^{n+1})
\end{multline}
Now we separately find bounds for each of the above terms. Applying the Cauchy-Schwarz and Youngs inequalities subsequently we have obtained the following results
\begin{equation}
\begin{split}
(\frac{M}{dt} (\textbf{U}^n_h +  \tilde{\textbf{U}}^n), \textbf{U}_h^{n+1}) & \leq \frac{\rho}{ dt}(\|\textbf{u}^n_h\|^2+ \|\textbf{u}_h^{n+1}\|^2+ \| \tilde{\textbf{u}}^n\|^2)\\
& \quad + \frac{1}{dt} (\|c^n_h\|^2+\|c^{n+1}_h\|^2+ \| \tilde{c}^n\|^2)\\
(\textbf{F}^{n+1}, \textbf{U}_h^{n+1}) + (\textbf{F}^{n+1}, \tilde{\textbf{U}}^{n+1}) & \leq \bar{\epsilon}_1 (\|f_1^{n+1}\|^2 +\|f_2^{n+1}\|^2+\|c^{n+1}\|^2)+\\
& \quad \frac{1}{2 \bar{\epsilon}_1}( \|\textbf{u}_h^{n+1}\|^2 + \| \tilde{\textbf{u}}^{n+1}\|^2+ \|g^{n+1}_h\|^2+\| \tilde{c}^n\|^2) \\
\end{split}
\end{equation}
To find an appropriate bound on $B(\cdot,\cdot; \cdot,\cdot)$ we have applied property $\textbf{(a)}$ on the trilinear term $c(\cdot,\cdot,\cdot)$ and assumptions $\textbf{(i)}-\textbf{(iv)}$ on the coefficients and finally arrived at the following
\begin{equation}
\begin{split}
B(\textbf{u}_h^n,\eta_h^n; \textbf{U}_h^{n+1}, \textbf{U}_h^{n+1}) 
& \geq 2 \eta_l  \| \nabla \textbf{u}_h^{n+1} \|^2 + D_l \| \nabla c^{n+1}\|^2 + \alpha \|c^{n+1}\|^2 \\
\end{split}
\end{equation}
Formation of $\eta_l$ is discussed in details in the proof of theorem 2. Now to estimate the remaining terms which are defined over each sub domain $\Omega_k$ we make use of an important observation: by the virtue of the choices of the finite element spaces $V_h$ and $Q_h$, we can clearly say that over each element sub domain every function belonging to that spaces and their first and second order derivatives all are bounded functions.
\begin{equation}
\begin{split}
 \sum_{k=1}^{n_{el}}(\tilde{u}_1^{n+1}, \frac{\partial}{\partial x}(2 \eta_h \frac{\partial u_{1h}^{n+1}}{\partial x})+\frac{\partial}{\partial y}(\eta_h( \frac{\partial u_{2h}^{n+1}}{\partial x}+ \frac{\partial u_{1h}^{n+1}}{\partial y})))_k & \leq \bar{C}_1 \|\tilde{u}_1^{n+1}\| \\
  \sum_{k=1}^{n_{el}}(\tilde{u}_2^{n+1}, \frac{\partial}{\partial x}(\eta_h( \frac{\partial u_{2h}^{n+1}}{\partial x}+ \frac{\partial u_{1h}^{n+1}}{\partial y}))+\frac{\partial}{\partial y}(2 \eta_h \frac{\partial u_{2h}^{n+1}}{\partial y}))_k & \leq \bar{C}_2 \|\tilde{u}_2^{n+1}\| \\
\sum_{k=1}^{n_{el}}(\tilde{c}^{n+1}, \nabla \cdot \tilde{\nabla} c_h^{n+1})_k & \leq \bar{C}_3 \|\tilde{c}^{n+1}\|
\end{split}
\end{equation} 
where $\bar{C}_i$ for i=1,2,3 are positive constants obtained after applying the above observation. Now combining all these estimated results in (20), multiplying the resulting equation by $dt$ and adding for $n=0$ to $N-1$ we prove the stability result (19).
\end{proof}

\section{Error estimates}
This section comprises detailed derivations of $ apriori$ and $aposteriori$ error estimates with respect to that full norm, based on which we have analyzed stability of the numerical method. Let us here introduce the norm definitions explicitly. For simplifying the notational expressions we specifically use the alphabets $\bar{\textbf{P}}$ and $\bar{\textbf{Q}}$ to denote the spaces $L^{\infty}(0,T; L^2(\Omega)) \bigcap L^{2}(0,T; H^1(\Omega)) $ and $L^{2}(0,T; L^2(\Omega))$ respectively, where for any $f \in \bar{\textbf{P}}$ and $g \in \bar{\textbf{Q}}$
\begin{center}
$\|f\|_{\bar{\textbf{P}}}^2  = \underset{0\leq n \leq N-1}{max} \|f^{n+1}\|^2 +dt \sum_{n=0}^{N-1} ( \| f^{n+1} \|^2  + \| \frac{\partial f}{\partial x}^{n+1} \|^2 + \| \frac{\partial f}{\partial y}^{n+1}\|^2)$\\
$\|g\|_{\bar{\textbf{Q}}}^2 = dt \sum_{n=0}^{N-1} \|g^{n+1}\|^2 dt $
\end{center}
Let us introduce the projection operator for each of these error components.\vspace{1 mm}\\
(P1) For any $\textbf{u} \in H^2(\Omega) \times H^2(\Omega) $ we assume that there exists an interpolation $I^h_{\textbf{u}}:  H^2(\Omega) \times H^2(\Omega) \longrightarrow  V_h \times V_h $ satisfying $b(\textbf{u}-I^h_{\textbf{u}}\textbf{u}, q_h)=0$ \hspace{2mm} $\forall q_h \in Q_h$. \vspace{2mm}\\
(P2) Let $I^h_p: H^1(\Omega) \longrightarrow Q_h$ be the $L^2$ orthogonal projection given by \\ $\int_{\Omega}(p-I^h_pp)q_h=0$  \hspace{1mm} $\forall q_h \in Q_h$ and for any $p \in H^1(\Omega)$ and\vspace{2mm}\\
(P3) Let $I^h_{c}: H^2(\Omega) \longrightarrow V_h$ be the $L^2$ orthogonal projection given by \\ $\int_{\Omega}(c-I^h_c c)d_h=0$ \hspace{1mm} $ \forall d_h \in V_h$ and for any $c \in H^2(\Omega)$. \vspace{2mm}\\
Let us denote the error by $\textbf{e}_{\textbf{U}}=(e_{\textbf{u}},e_p,e_c)$, where the components can be separately split into interpolation part, $E^I$ and auxiliary part, $E^A$ as follows:  \vspace{1mm}\\
\begin{equation}
\begin{split}
e_{\textbf{u}} & =(\textbf{u}-\textbf{u}_h)=(\textbf{u}-I^h_{\textbf{u}} \textbf{u})+(I^h_{\textbf{u}} \textbf{u}-\textbf{u}_h)= E^{I}_{\textbf{u}}+ E^{A}_{\textbf{u}}\\
e_{p} & =E^{I}_{p}+ E^{A}_{p} \hspace{2mm} and  \hspace{2mm} e_{c}=E^{I}_{c}+ E^{A}_{c}
\end{split}
\end{equation}
The standard \textbf{interpolation estimation} result \cite{4} is given as: for any exact solution with regularity upto (r+1)
\begin{equation}
\|v-I^h_v v\|_l = \|E^I_v\|_l \leq C(p,\Omega) h^{r+1-l} \|v\|_{r+1} 
\end{equation}
where l ($\leq r+1$) is a positive integer and C is a constant depending on m and the domain. 
The following results \cite{34} associated with projection operators play important role in error estimations.
\begin{result}
 For any interpolation error $E^I$
\begin{equation}
(\frac{\partial}{\partial t} E^{I}, v_h)=0 \hspace{2mm} \forall v_h \in V_h
\end{equation}
\end{result}
\begin{result}
 For any given auxiliary error $E^{A,n}$ and unknown $E^{A,n+1}$
\begin{equation}
(\frac{\partial}{\partial t} E^{A,n}, E^{A,n,\theta}) \geq  \frac{1}{2 dt} (\|E^{A,n+1}\|^2- \|E^{A,n}\|^2)
\end{equation}
\end{result}

\subsection{Apriori error estimation}
The whole derivation  involves a separate estimation of $auxiliary$ error bound in first part and later using that result and standard interpolation estimate we ultimately derive $apriori$ error bound.
\begin{theorem} (Auxiliary error estimate) \hspace{1mm} Assuming the viscosity,diffusion, density and reaction coefficients satisfying the assumptions \textbf{(i)}-\textbf{(iii)} and considering adequately small time step $dt(>0)$, then for sufficiently regular continuous solutions $(\textbf{u},p,c)$ satisfying the assumptions \textbf{(iv)}-\textbf{(v)} and the computed solutions $(\textbf{u}_h,p_h,c_h) \in$ $V_h \times V_h \times Q_h \times V_h$ satisfying (19),  there exists a constant $\tilde{C}(\textbf{U})$, depending upon the continuous solution such that
\begin{center}
$\|E^A_{\textbf{u}}\|^2_{\bar{\textbf{P}}}+ \|E^A_p\|_{\bar{\textbf{Q}}}^2  + \|E^A_{c}\|^2_{\bar{\textbf{P}}} \leq \tilde{C}(\textbf{U}) (h^2+ dt^{2})$
\end{center}
\end{theorem}
\begin{proof}
Derivation of this auxiliary error bound goes in two parts: in the first part bound has been found out for auxiliary error corresponding to the velocity and concentration variables, whereas the second part comprises of estimation of auxiliary pressure term and the final result has been achieved on combining these two estimates. \vspace{2mm} \\
\textbf{First part:} Expanding the terms after subtracting fully-discrete $ASGS$ stabilized formulation (19) from fully-discrete version of weak formulation (8) for the exact solution we have $\forall \hspace{1mm} \textbf{V}_h \in V_h \times V_h \times Q_h \times V_h$
\begin{multline}
(M \partial_t(\textbf{U}^{n}-\textbf{U}^{n}_h),\textbf{V}_{h}) + \{ B(\textbf{u}^{n},  \mu^{n} ;\textbf{U}^{n+1}, \textbf{V}_h) -B(\textbf{u}_h^{n},  \mu_h^{n};\textbf{U}^{n+1}_h, \textbf{V}_h)\}+ \\
 \sum_{k=1}^{n_{el}}(\tau_k'(M\partial_t (\textbf{U}^n-\textbf{U}^n_h)+
 \mathcal{L}(\textbf{u}^{n},  \mu^{n} ;\textbf{U}^{n+1})-\mathcal{L}(\textbf{u}^{n}_h,  \mu_h^{n};\textbf{U}^{n+1}_h)),
-\mathcal{L}^* (\textbf{u}_h,  \mu_h;\textbf{V}_h))_{\Omega_k} \\
+\sum_{k=1}^{n_{el}}((I-\tau_k^{-1}\tau_k')(M \partial_t(\textbf{U}^{n}-\textbf{U}^{n}_h) + 
 \mathcal{L}(\textbf{u}^{n}, \mu^{n} ;\textbf{U}^{n+1})-\mathcal{L}(\textbf{u}^{n}_h,  \mu_h^{n} ;\textbf{U}^{n+1}_h)), -\textbf{V}_h)_{\Omega_k}\\
+ \sum_{k=1}^{n_{el}} (\tau_k^{-1}\tau_k' \textbf{d}, \textbf{V}_h)_{\Omega_k}+\sum_{k=1}^{n_{el}}(\tau_k' \textbf{d},-\mathcal{L}^*(\textbf{u}_h,  \mu_h ; \textbf{V}_h))_{\Omega_k}=(\textbf{TE}^{n+1}, \textbf{V}_h)
\end{multline}
On expanding out the terms in the matrix $\textbf{d}$= $\{\sum_{i=0}^{N}(\frac{1}{dt}M\tau_k')^i)(M\partial_t (\textbf{U}^n-\textbf{U}^n_h) + \mathcal{L}(\textbf{u}^{n},\mu^{n} ;\textbf{U}^{n+1})-
\mathcal{L}(\textbf{u}_h^{n}, \mu_h^{n} ;\textbf{U}_h^{n+1}) \}$.
 Now we apply (26) on each of the error terms appearing in (30), later we use result in (28) and property (P1) of projection operator for interpolation errors. Next to that rearranging the resultant equation in the following manner we have $\forall  \textbf{V}_h \in V_h \times V_h \times Q_h \times V_h$
\begin{multline}
 (M \frac{E^{A,n+1}_{\textbf{U}}-E^{A,n}_{\textbf{U}}}{dt}, \textbf{V}_h) + a_{NL}^F(\mu^n; E^{A,n+1}_{\textbf{u}}, \textbf{v}_h)+a_L^T(E^{A,n+1}_c,d_h)+ \int_{\Omega} \mu^n (E^{A,n+1}_{\textbf{u}})^2 \\
 = b(\textbf{v}_h, E^{A,n+1}_p) -b(E^{A,n+1}_{\textbf{u}},q_h) +b(\textbf{v}_h, E^{I,n+1}_p) -a_L^T(E^{I,n+1}_c,d_h)-  \\
\quad \{ a_{TL}^F(\textbf{u}^n, \textbf{u}^{n+1},\textbf{v}_h)-a_{TL}^F(\textbf{u}^n_h, \textbf{u}^{n+1}_h,\textbf{v}_h) \} + \int_{\Omega} \mu^n E^{A,n+1}_{\textbf{u}} \cdot E^{A,n+1}_{\textbf{u}}-\\
\quad \{a_{TL}^T(\textbf{u}^n,c^{n+1},d_h)- a_{TL}^T(\textbf{u}_h^n,c^{n+1}_h,d_h)\} - a_{NL}^F(\mu^n; E^{I,n+1}_{\textbf{u}}, \textbf{v}_h) -\\
\quad a_{NL}^F(\mu^n-\mu_h^n; \textbf{u}^{n+1}_h, \textbf{v}_h) -\sum_{k=1}^{n_{el}} (I_1^k+I_2^k+I_3^k+I_4^k)+ (\textbf{TE}^{n+1}, \textbf{V}_h) 
\end{multline}
where $I_i^k$, for $i=1,...,4$ denote the last four terms in the $LHS$ of (30) in short. The purpose of this rearrangement lies in combining the lower and upper bounds (to be found) for the terms in left hand side and right hand side of (31) respectively to arrive at the desired estimate. As the equation (31) holds for all $\textbf{V}_h \in V_h \times V_h \times  Q_h \times V_h$ we in particular replace $\textbf{v}_{h}, q_h,d_h$ by $E^{A,n+1}_{\textbf{u}},E^{A,n+1}_{p},E^{A,n+1}_{c}$ respectively. \vspace{1mm}\\
Let us begin the estimation by finding out a lower bound for the second term in the left hand side of (31).
\begin{equation}
\begin{split}
& a_{NL}^F(\mu^n; E^{A,n+1}_{\textbf{u}},  E^{A,n+1}_{\textbf{u}}) \\
 & = \int_{\Omega} \mu(c^n,J_2^n) \{2 (\frac{\partial E^{A,n+1}_{u1}}{\partial x})^2 +  ( \frac{\partial E^{A,n+1}_{u1}}{\partial y})^2+ (\frac{\partial E^{A,n+1}_{u2}}{\partial x})^2+ 2 (\frac{\partial E^{A,n+1}_{u2}}{\partial y})^2 \} \\
& \geq  \mu_l \mid E^{A,n+1}_{\textbf{u}} \mid_1^2
\end{split}
\end{equation}
where
\begin{equation}
\mu(c^n,J_2^n)=\frac{(\frac{\sqrt{\tau_y}}{\sqrt{2}} + \sqrt{\eta(c^n)} \{ 2 (\frac{\partial u_1^n}{\partial x_1})^2 + 2 (\frac{\partial u_2^n}{\partial x_2})^2 + (\frac{\partial u_1^n}{\partial x_2}+ \frac{\partial u_2^n}{\partial x_1})^2 \}^{\frac{1}{4}})^2}{\{ 2 (\frac{\partial u_1^n}{\partial x_1})^2 + 2 (\frac{\partial u_2^n}{\partial x_2})^2 + (\frac{\partial u_1^n}{\partial x_2}+ \frac{\partial u_2^n}{\partial x_1})^2 \}^{\frac{1}{2}}} 
\end{equation}
Now applying assumption $\textbf{(i)}$ on the viscosity coefficient $\eta(c)$ associated with the expression of $\mu(c,J_2)$ and assumption \textbf{(v)} on the various terms of $J_2$ we have the lower bound $\mu_l$ as follows:
\begin{center}
$\mu_l$= $\frac{ \{ \frac{\sqrt{\tau_y}}{\sqrt{2}}+ \sqrt{\eta_l} (C^u_l)^{\frac{1}{4}} \}^2}{(C^u_u)^{\frac{1}{2}}}$ 
\end{center}
where $C^u_l$ and $C^u_u$ denote respectively the sum of the lower bounds and upper bounds of all the combinations of $(\frac{\partial u_i}{\partial x_j} + \frac{\partial u_j}{\partial x_i})^2$  for $i,j=1,2$.  Applying result (29) on the first term and assumption $\textbf{(ii)}$ on the diffusion coefficients of the second term in the left hand side we have
\begin{equation}
\begin{split}
LHS & \geq \frac{\rho}{2 dt}  (\|E^{A,n+1}_{\textbf{u}}\|^2- \|E^{A,n}_{\textbf{u}}\|^2)+ \frac{1}{2 dt}(\|E^{A,n+1}_c\|^2- \|E^{A,n}_c\|^2)+ \mu_l  \| E^{A,n+1}_{\textbf{u}}\|_1^2 \\
& \quad +D_l \mid E^{A,n+1}_c \mid_1^2 + \alpha \|E^{A,n+1}_c\|^2
\end{split}
\end{equation}
where $D_l$= $ \underset{i=1,2}{min} \hspace{1mm} \underset{\Omega}{inf} D_i $. \\
Now our aim is to find upper bounds for the terms in the $RHS$ in (31).  Post substitution of $\textbf{V}_h$ we can see that the first two bilinear terms $b(\cdot,\cdot)$ get canceled out with each other and estimations of the remaining bilinear terms are quite straight-forward by applying the Cauchy-Schwarz inequality, Young's Inequality, standard interpolation estimate (24) and assumptions \textbf{(ii)} on the diffusion coefficients wherever they are required. Hence passing over the detailed derivations of those terms we only bring up here the final product.
\begin{equation}
\begin{split}
b(E^{A,n+1}_{\textbf{u}},E^{I,n+1}_p) &  \leq  \frac{1}{2 \epsilon_1} \mid E^{A,n+1}_{\textbf{u}} \mid_1^2+ \epsilon_1 C^2 h^2(\frac{1+\theta}{2}\| p^{n+1}\|_1 + \frac{1-\theta}{2}\| p^n \|_1)^2 \\
a_{L}^T(E^{I,n+1}_c, E^{A,n+1}_c) 
& \leq \frac{D_m}{2 \epsilon_2}  \mid E^{A,n+1}_{c} \mid_1^2+ \frac{D_m \epsilon_2}{2} C^2 h^2 (\frac{1+\theta}{2}\| c^{n+1} \|_2+\frac{1-\theta}{2}\| c^n \|_2)^2
\end{split}
\end{equation}
Again applying assumption $\textbf{(i)}$ on the viscosity coefficient and assumption \textbf{(v)} on the various terms of $J_2$ we can find an upper bound $\mu_u$ of $\mu(c,J_2)$ as follows:
\begin{center}
$\mu_u$= $\frac{ \{ \frac{\sqrt{\tau_y}}{\sqrt{2}}+ \sqrt{\eta_l} (C^u_u)^{\frac{1}{4}} \}^2}{(C^u_l)^{\frac{1}{2}}}$ 
\end{center}
Using $Poincare$ inequality
\begin{equation}
\int_{\Omega} \mu^n E^{A,n+1}_{\textbf{u}} \cdot E^{A,n+1}_{\textbf{u}} \leq \mu_u C_P \mid E^{A,n+1}_{\textbf{u}} \mid_1^2
\end{equation}
where $C_P$ is the $Poincare$ constant. Now the estimation of the trilinear terms $a_{TL}^F(\cdot,\cdot,\cdot)$  begins with splitting the error terms occurred during rearranging them. Applying properties \textbf{(a)} and $\textbf{(b)}$ given in (9) and the standard interpolation estimate (27) on each of the terms consecutively we have the following result:
\begin{equation}
\begin{split}
& \{ a_{TL}^F(\textbf{u}^n, \textbf{u}^{n+1},\textbf{v}_h)-a_{TL}^F(\textbf{u}^n_h, \textbf{u}^{n+1}_h,\textbf{v}_h) \} \\
& \leq \mid  a_{TL}^F(\textbf{u}^n, E^{I,n+1}_{\textbf{u}}, E^{A,n+1}_{\textbf{u}}) \mid + \mid  a_{TL}^F(\textbf{u}^n, E^{A,n+1}_{\textbf{u}}, E^{A,n+1}_{\textbf{u}}) \mid + \mid  a_{TL}^F(E^{I,n}_{\textbf{u}}, \textbf{u}^{n+1}, E^{A,n+1}_{\textbf{u}}) \mid \\
& \quad  + \mid  a_{TL}^F(E^{I,n}_{\textbf{u}}, E^{I,n}_{\textbf{u}}, E^{A,n+1}_{\textbf{u}}) \mid + \mid  a_{TL}^F(E^{I,n}_{\textbf{u}}, E^{A,n}_{\textbf{u}}, E^{A,n+1}_{\textbf{u}}) \mid +  \mid  a_{TL}^F(E^{A,n}_{\textbf{u}}, \textbf{u}^{n+1}, E^{A,n+1}_{\textbf{u}}) \mid \\
& \quad +  \mid  a_{TL}^F(E^{A,n}_{\textbf{u}},E^{I,n+1}_{\textbf{u}}, E^{A,n+1}_{\textbf{u}}) \mid +  \mid  a_{TL}^F(E^{A,n}_{\textbf{u}},E^{A,n+1}_{\textbf{u}}, E^{A,n+1}_{\textbf{u}}) \mid \\
& \leq h^2 \frac{C_2 \epsilon_3 }{2} \{ h^2( \|\textbf{u}^{n}\|+1)+1 \} (\frac{1+\theta}{2} \|\textbf{u}^{n+1}\|_2 + \frac{1-\theta}{2} \|\textbf{u}^{n}\|_2)^2 +\frac{3 C_2}{ 2 \epsilon_3} \|E^{A,n+1}_{\textbf{u}}\|_1^2 \\
\end{split}
\end{equation}
Estimation of another trilinear term $a_{TL}^T(\cdot,\cdot,\cdot)$ follows the same way in the beginning but final result is reached in the following as the consequence of exercising the Sobolev, the Cauchy-Schwarz and the Young inequalities subsequently.
\begin{equation}
\begin{split}
& \{a_{TL}^T(\textbf{u}^n,c^{n+1},E^{A,n+1}_c)- a_{TL}^T(\textbf{u}_h^n,c^{n+1}_h,E^{A,n+1}_c)\} \\
& \leq \mid a_{TL}^T(\textbf{u}^n, E^{I,n+1}_c,E^{A,n+1}_c) \mid + \mid a_{TL}^T(\textbf{u}^n, E^{A,n+1}_c,E^{A,n+1}_c) \mid + \mid a_{TL}^T(E^{I,n}_{\textbf{u}}, c^{n+1},E^{A,n+1}_c) \mid \\
& \quad +  \mid a_{TL}^T(E^{A,n}_{\textbf{u}}, c^{n+1},E^{A,n+1}_c) \mid +  \mid a_{TL}^T(E^{I,n}_{\textbf{u}}, E^{I,n+1}_c,E^{A,n+1}_c) \mid +  \mid a_{TL}^T(E^{I,n}_{\textbf{u}}, E^{A,n+1}_c,E^{A,n+1}_c) \mid\\
& \quad +  \mid a_{TL}^T(E^{A,n}_{\textbf{u}}, E^{I,n+1}_c,E^{A,n+1}_c) \mid +  \mid a_{TL}^T(E^{I,n}_{\textbf{u}}, E^{A,n+1}_c,E^{A,n+1}_c) \mid +  \mid a_{TL}^T(E^{A,n}_{\textbf{u}}, E^{A,n+1}_c,E^{A,n+1}_c) \mid \\
& \leq  \frac{C^2 h^2}{\epsilon_4} \{ 1+ C^2 h^2 \|\textbf{u}^n\|_2^2 \}(\frac{1+\theta}{2} \|c^{n+1}\|_2 + \frac{1-\theta}{2} \|c^{n}\|_2)^2 + 4\epsilon_4 \|E^{A,n+1}_c\|^2 + \frac{\bar{C_1}+ \bar{C_2}}{2}  \|E^{A,n+1}_c\|_1^2 \\
& \quad + \frac{C_1'}{2 \epsilon_4} \|\frac{\partial E^{A,n+1}_c}{\partial x}\|^2+  \frac{C_2'}{2 \epsilon_5} \|\frac{\partial E^{A,n+1}_c}{\partial y}\|^2 + \frac{C^4 h^4}{2 \epsilon_4} \{  \|c^{n+1}\|_2^2 +  (\frac{1+\theta}{2} \|\textbf{u}^{n+1}\|_2+  \frac{1-\theta}{2} \|\textbf{u}^{n}\|_2)^2 \}
\end{split}
\end{equation}
On expanding the estimation of the term $a_{NL}^F(\cdot; \cdot,\cdot)$ can be easily carried out with the help of above mentioned standard inequalities and the assumptions made by  us.
\begin{equation}
\begin{split}
a_{NL}^F(\mu^n; E^{I,n+1}_{\textbf{u}}, E^{A,n+1}_{\textbf{u}}) & \leq \mu_u \{ \frac{C^2 h^2}{\epsilon_5} (\frac{1+\theta}{2} \|\textbf{u}^{n+1}\|_2 + \frac{1-\theta}{2} \|\textbf{u}^{n}\|_2)^2  + \epsilon_5 \mid E^{A,n+1}_{\textbf{u}}\mid_1^2 \} \\
a_{NL}^F(\mu^n-\mu_h^n; \textbf{u}^{n+1}_h, E^{A,n+1}_{\textbf{u}}) & \leq \mid a_{PL}(\mu^{n}; \textbf{u}^{n+1}, E^{A,n+1}_\textbf{u}) - a_{PL}(\mu^{n}; E_\textbf{u}^{I,n+1}, E^{A,n+1}_\textbf{u}) \\
& \quad - a_{PL}(\mu^{n}; E_\textbf{u}^{A,n+1}, E^{A,n+1}_\textbf{u}) \mid  \\
& \leq  \mu_u \{ \frac{C^2 h^2}{\epsilon_5} (\frac{1+\theta}{2} \|\textbf{u}^{n+1}\|_2 + \frac{1-\theta}{2} \|\textbf{u}^{n}\|_2)^2  + \epsilon_5 \mid E^{A,n+1}_{\textbf{u}}\mid_1^2 \}
\end{split}
\end{equation}
Applying the Cauchy-Schwarz and Young's inequality once again to estimate the truncation error terms as follows:
\begin{equation}
\begin{split}
(\textbf{TE}^{n+1}, E^{A,n+1}_{\textbf{U}})
& \leq \frac{ \epsilon_6}{2} \|\textbf{TE}^{n+1}\|^2 + \frac{1}{2 \epsilon_6} (\|E^{A,n+1}_{\textbf{u}}\|^2 +\|E^{A,n+1}_{c}\|^2)
\end{split}
\end{equation}
Now we estimate the remaining terms which are defined on the sub domain $\Omega_k$ in a slightly different manner. Here we make use of an important observation which says, the choices of the finite element spaces $V_h$ and $Q_h$ enable us to assume every function belonging to that spaces along with their first and second order derivatives as bounded functions over each element sub domain.
\begin{equation}
\begin{split}
I_1^k & =\tau_k' (M \partial_t (\textbf{U}^n-\textbf{U}^n_h)+ \mathcal{L}(\textbf{u}^n, \mu^n ; \textbf{U}^{n+1})- \mathcal{L}(\textbf{u}_h^n, \mu_h^n; \textbf{U}_h^{n+1})), -\mathcal{L}^*(\textbf{u}_h, \mu_h; E^{A,n+1}_{\textbf{U}}))_{k} \\
& =\tau_{1k}' ( \partial_t (\textbf{u}^n-\textbf{u}^n_h)+ \rho ((\textbf{u}^n \cdot \nabla) \textbf{u}^{n+1}-(\textbf{u}^n_h \cdot \nabla) \textbf{u}^{n+1}_h) + \nabla (p^{n+1}-p^{n+1}_h)  - \nabla \cdot 2 \\
& \quad (\mu \textbf{D}(\textbf{u}^{n+1})-  \mu_h \textbf{D}(\textbf{u}^{n+1}_h)) , \rho(\textbf{u}^n \cdot \nabla) E^{A,n+1}_{\textbf{u}}- \nabla \cdot 2 \mu \textbf{D}(E^{A,n+1}_{\textbf{u}})+  \nabla E^{A,n+1}_p)_{k}+ \\
& \quad  \tau_{2k} (\nabla \cdot (\textbf{u}^{n+1}-\textbf{u}_h^{n+1}), \nabla \cdot E^{A,n+1}_{\textbf{u}})_{k} + \tau_{3k}' (\partial_t (c^{n+1}-c_h^{n+1})-  \nabla \cdot \tilde{\nabla} (c^{n+1}-c^{n+1}_h) \\
& \quad + \textbf{u}^n \cdot \nabla c^{n+1}- \textbf{u}^n_h \cdot \nabla c_h^{n+1} + \alpha (c^{n+1}-c^{n+1}_h), \nabla \cdot \tilde{\nabla} E^{A,n+1}_c + \textbf{u}^n \cdot \nabla E^{A,n+1}_c- \\
& \quad  \alpha E^{A,n+1}_c)_{k}
\end{split}
\end{equation}
where $ \textbf{D}(\textbf{u})$ denotes the deformation tensor. Now subsequently applying error splitting, the above observation on the terms belonging to the finite element spaces, the standard interpolation estimate (27) and the assumption $\textbf{(iv)}$ on the continuous solutions we have reached at the following result.
\begin{equation}
\begin{split}
\sum_{k=1}^{n_{el}} I_1^k &  \leq \mid \tau_1 \mid \bar{C}_1^1(h,\textbf{u}^{n+1},p^{n+1}) + h^2 \mid \tau_2 \mid  \bar{C}^1_2(\textbf{u}^{n+1}) + \mid \tau_3 \mid \bar{C}_3^1(h,c^{n+1})
\end{split}
\end{equation}
where the parameters $\bar{C}_i^1$ (for $i=1,2,3$) denote summations of positive constants, obtained as a consequence of exercising the observation on $E^{A,n+1}_{\textbf{U}}$ and its derivatives and assumption \textbf{(iv)} on the exact solutions.
\begin{equation}
\begin{split}
I_2^k & =((I- \tau_k^{-1} \tau_k') (M \partial_t (\textbf{U}^n-\textbf{U}^n_h)+ \mathcal{L}(\textbf{u}^n, \mu^n; \textbf{U}^{n+1})- \mathcal{L}(\textbf{u}_h^n, \mu_h^n; \textbf{U}_h^{n+1})),-E^{A,n+1}_{\textbf{U}})_{k} \\
\end{split}
\end{equation}
We see that the second $I^k$ term looks alike to the previous one and therefore we do not go into its expansion but let us have a glance into the coefficient term involving stabilization parameters.
\[
(I- \tau_k^{-1} \tau_k')=
  \begin{bmatrix}
 (1-\frac{dt}{dt+\rho \tau_{1k}})I_2 & 0 & 0 \\
 0 & 0 & 0 \\
   0 & 0 &  (1-\frac{dt}{dt+\tau_{3k}})
  \end{bmatrix} 
  =
    \begin{bmatrix}
 \frac{\rho \tau_{1k}}{dt+\rho \tau_{1k}} I_2 & 0 & 0 \\
 0 & 0 & 0 \\
   0 & 0 &  \frac{\tau_{3k}}{dt+\tau_{3k}}
  \end{bmatrix} 
\]
On expansion of the next term
\begin{equation}
\begin{split}
I_3^k & = (\tau_k^{-1}\tau_k' \textbf{d}, E^{A,n+1}_{\textbf{U}})_{k} \\
&= (\tau_k^{-1}\tau_k' \sum_{i=0}^{N} (\frac{1}{dt}M\tau_k')^i (M\partial_t (\textbf{U}^n-\textbf{U}^n_h) +  \mathcal{L}(\textbf{u}^n, \mu^n; \textbf{U}^{n+1})- \mathcal{L}(\textbf{u}_h^n, \mu_h^n; \textbf{U}_h^{n+1})),\\ & \quad E^{A,n+1}_{\textbf{U}})_{k} \\
& \leq  (\tau_k^{-1}\tau_k' \sum_{i=0}^{\infty} ((\frac{1}{dt}M\tau_k')^i)(M\partial_t (\textbf{U}^n-\textbf{U}^n_h) +  \mathcal{L}(\textbf{u}^n, \mu^n; \textbf{U}^{n+1})- \mathcal{L}(\textbf{u}_h^n, \mu_h^n; \textbf{U}_h^{n+1})), \\
& \quad E^{A,n+1}_{\textbf{U}})_{k} \\
& = \frac{\rho \tau_{1k}}{(dt+\rho \tau_{1k})}  ( \partial_t (\textbf{u}^n-\textbf{u}^n_h)+ \rho ((\textbf{u}^n \cdot \nabla) \textbf{u}^{n+1}-(\textbf{u}^n_h \cdot \nabla) \textbf{u}^{n+1}_h) + \nabla (p^{n+1}-p^{n+1}_h)  \\
& \quad  - \nabla \cdot 2(\mu \textbf{D}(\textbf{u}^{n+1})-  \mu_h \textbf{D}(\textbf{u}^{n+1}_h)) ,  E^{A,n+1}_{\textbf{u}})_{k}+ \frac{\tau_{3k}}{(dt+\tau_{3k})}(\partial_t (c^{n+1}-c_h^{n+1})- \nabla \cdot\\
& \quad  \tilde{\nabla} (c^{n+1}-c^{n+1}_h) +  \textbf{u}^n \cdot \nabla c^{n+1} -\textbf{u}^n_h \cdot \nabla c_h^{n+1} + \alpha (c^{n+1}-c^{n+1}_h), E^{A,n+1}_c)_{k} 
\end{split}
\end{equation}
For $dt>0$, $\frac{\rho \tau_{1k}}{dt+\rho \tau_{1k}} < 1$ and $\frac{\tau_{3k}}{dt+ \tau_{3k}} < 1$, which implies $\frac{\rho \tau_{1k}'}{dt} < 1$ and $\frac{\tau_{3k}'}{dt} < 1$ and therefore the series $\sum_{i=1}^{\infty}(\frac{\rho}{dt} \tau_{1k}' )^i$ and $\sum_{i=1}^{\infty}(\frac{1}{dt} \tau_{3k}' )^i$ converges to $\frac{\rho \tau_{1k}' }{(dt- \tau_{1k}' )}=\frac{\rho \tau_{1k}}{dt}$ and $\frac{\tau_{3k}' }{(dt- \tau_{3k}' )}=\frac{\tau_{3k}}{dt}$ respectively.
\begin{equation}
\begin{split}
I_4^k & =(\tau_k' \textbf{d},-\mathcal{L}^*(\textbf{u}_h, \eta_h; E^{A,n+1} _{\textbf{U}}))_{k} \\
& \leq \frac{\rho \tau_{1k}^2}{(dt+\rho \tau_{1k})}  ( \partial_t (\textbf{u}^n-\textbf{u}^n_h)+ \rho ((\textbf{u}^n \cdot \nabla) \textbf{u}^{n+1}-(\textbf{u}^n_h \cdot \nabla) \textbf{u}^{n+1}_h) + \nabla (p^{n+1}-p^{n+1}_h) - \\
& \quad  \nabla \cdot 2(\mu \textbf{D}(\textbf{u}^{n+1})-  \mu_h \textbf{D}(\textbf{u}^{n+1}_h)) ,  \rho(\textbf{u}^n \cdot \nabla) E^{A,n+1}_{\textbf{u}}+ \nabla E^{A,n+1}_p- \nabla \cdot 2 \mu \textbf{D}(E^{A,n+1}_{\textbf{u}}))_{k}   \\
& \quad + \frac{\tau_{3k}^2}{(dt+\tau_{3k})}(\partial_t (c^{n+1}-c_h^{n+1})- \nabla \cdot \tilde{\nabla} (c^{n+1}-c^{n+1}_h) + \textbf{u}^n \cdot \nabla c^{n+1}- \textbf{u}^n_h \cdot \nabla c_h^{n+1} \\
& \quad  + \alpha (c^{n+1}-c^{n+1}_h), \nabla \cdot \tilde{\nabla} E^{A,n+1}_c + \textbf{u}^n \cdot \nabla E^{A,n+1}_c - \alpha E^{A,n+1}_c)_{k}
\end{split}
\end{equation}
Clearly on expansion $I_3^k$ and $I_4^k$ look same as $I_2^k$ and $I_1^k$ respectively and hence the estimated results of the remaining terms are as follows:
\begin{equation}
\begin{split}
\sum_{k=1}^{n_{el}} I_2^k & \leq \mid \tau_1 \mid \bar{C}_1^2(h,\textbf{u}^{n+1},p^{n+1}) + \mid \tau_3 \mid \bar{C}_3^2(h,c^{n+1}) \\
\sum_{k=1}^{n_{el}} I_3^k  & \leq \mid \tau_1 \mid \bar{C}_1^3(h,\textbf{u}^{n+1},p^{n+1})  + \mid \tau_3 \mid \bar{C}_3^3(h,c^{n+1})\\
\sum_{k=1}^{n_{el}} I_4^k  & \leq \mid \tau_1 \mid \bar{C}_1^4(h,\textbf{u}^{n+1},p^{n+1})  + \mid \tau_3 \mid \bar{C}_3^4(h,c^{n+1})
\end{split}
\end{equation}
where the parameters $\bar{C}_1^i, \bar{C}_3^i$ for $i=2,3,4$ are obtained similarly. The term wise estimations are completed here. Now combining the results (34)-(46) in (31), arranging them suitably, multiplying both sides by 2$dt$ and taking summation over the time steps for n=0,1,...,$(N-1)$ to both the sides subsequently we finally arrive at the following:
\begin{multline}
\rho \sum_{n=0}^{N-1} (\|E^{A,n+1}_{\textbf{u}}\|^2- \|E^{A,n}_{\textbf{u}}\|^2)+   \sum_{n=0}^{N-1} (\|E^{A,n+1}_c\|^2- \|E^{A,n}_c\|^2)+ (2 \mu_l-\frac{1}{\epsilon_1}-\frac{3C_2}{\epsilon_3}-2C_2'-\\
\quad 4 \mu_u \epsilon_5-2 \mu_u C_P-\frac{1}{\epsilon_6})  \sum_{n=0}^{N-1} \mid E^{A,n+1}_{\textbf{u}}\mid_1^2 dt+ (2 \mu_l-\frac{3C_2}{\epsilon_3}-2C_2'-\frac{1}{\epsilon_6})  \sum_{n=0}^{N-1} \| E^{A,n+1}_{\textbf{u}}\|_1^2 dt+  \\
\quad \{D_l- \frac{D_m}{\epsilon_2}- (\bar{C}_1+\bar{C}_2) \epsilon_4-\frac{1}{\epsilon_6} \}  \sum_{n=0}^{N-1} \mid E^{A,n+1}_{c}\mid_1^2 dt + (2 \alpha- 4 \epsilon_4-\frac{1}{\epsilon_6})  \sum_{n=0}^{N-1} \| E^{A,n+1}_{c}\|_1^2 dt\\
\leq h^2  \sum_{n=0}^{N-1} [2\epsilon_1 C^2 (\frac{1+\theta}{2} \|p^{n+1}\|_1+\frac{1-\theta}{2} \|p^{n}\|_1)^2+C^2( D_m \epsilon_2+ \frac{\bar{C}_1+\bar{C}_2}{\epsilon_4}+ \frac{C^2 h^2}{\epsilon_4} \|\textbf{u}^n\|_2^2) \\
\quad (\frac{1+\theta}{2} \|c^{n+1}\|_2+\frac{1-\theta}{2} \|c^{n}\|_2)^2 + (h^2C_2 \epsilon_3+C_2 \epsilon_3 + \frac{C^2}{\epsilon_4} \|c^{n+1}\|_2^2+\frac{2\mu_u C^2}{\epsilon_5}) \\
\quad (\frac{1+\theta}{2} \|\textbf{u}^{n+1}\|_2+\frac{1-\theta}{2} \|\textbf{u}^{n}\|_2)^2 ] dt +  \mid \tau_1 \mid \sum_{n=0}^{N-1}  \sum_{i=1}^3
 \bar{C}_1^i(h,\textbf{u}^{n+1},p^{n+1}) dt+\\
  \mid \tau_2 \mid h^2  \sum_{n=0}^{N-1} \bar{C}^1_2(\textbf{u}^{n+1}) dt+ \mid \tau_3 \mid \sum_{n=0}^{N-1} \sum_{i=1}^3  \bar{C}_1^i(h,\textbf{u}^{n+1},p^{n+1}) dt +\epsilon_6 \sum_{n=0}^{N-1} \|\textbf{TE}^{n+1}\|^2 dt
\end{multline}
The presence of the arbitrary parameters $\epsilon_i$ for $i=1,2,...,6$ and other positive constants in the left hand side of (47) enable us to choose their values in such a way so that we can make all the coefficients positive. Now we divide both the sides with the minimum of the coefficients in $LHS$. Applying the result (12) on the truncation error terms we finally arrive at the following expression using the fact that $\tau_1$ and $\tau_3$ are of order $h^2$:
\begin{equation}
\boxed{ \|E^{A}_{\textbf{u}}\|_{\bar{\textbf{P}}}^2+ \|E^{A}_{c}\|_{\bar{\textbf{P}}}^2 \leq C(T,\textbf{u},p,c) (h^2+dt^{2})}
\end{equation}
Here we use the fact: $\sum_{n=0}^{N-1} \int_{t^n}^{t^{n+1}} M dt \leq M T  $. This completes the first part of the proof. \vspace{2mm}\\
\textbf{Second part:} For estimating auxiliary pressure error we begin with the Galerkin orthogonality result for the flow problem as follows:
\begin{multline}
 b(\textbf{v}_h,p-I^h_p p)+ b(\textbf{v}_h,I^h_p p- p_h)=(\partial_t E^A_{\textbf{u}}, \textbf{v}_h) + a^F_{TL}(\textbf{u},\textbf{u},\textbf{v}_h)
-a^F_{TL}(\textbf{u}_h,\textbf{u}_h,\textbf{v}_h) \\
+ a_{NL}^F(\mu ;\textbf{u},\textbf{v}_h)-a_{NL}^F(\mu_h ;\textbf{u}_h,\textbf{v}_h)
\end{multline}
Applying the inclusion $\bigtriangledown \cdot V_h \subset Q_h$ and the property of the $L^2$ orthogonal projection of $I^h_p p$ we have 
\begin{equation}
b(\textbf{v}_h, p-I^h_p p)= \int_{\Omega}(p-I^h_p p)(\bigtriangledown \cdot \textbf{v}_h)=0
\end{equation}
Now according to discrete inf-sup condition 
\begin{equation}
\begin{split}
\|I^h_p p-p_h\|_{\bar{\textbf{P}}}^2 = \|E_p^A\|_{\bar{\textbf{P}}}^2 
& =  \sum_{n=0}^{N-1}  \|E_p^{A,n+1}\|^2 dt \\
& \leq  \sum_{n=0}^{N-1} \underset{\textbf{v}_h}{sup} \frac{b(\textbf{v}_h, E_p^{A,n+1})}{\|\textbf{v}_h\|_1} dt
\end{split}
\end{equation}
Applying (50) on (49) we have
\begin{equation}
\begin{split}
\sum_{n=0}^{N-1} b(\textbf{v}_h, E_p^{A,n+1}) dt & = \sum_{n=0}^{N-1} \{ (\frac{E^{A,n+1}_{\textbf{u}}-E^{A,n}_{\textbf{u}}}{dt},\textbf{v}_{h})+ a^F_{TL}(\textbf{u}^n,\textbf{u}^{n+1},\textbf{v}_h) - \\
& \quad a^F_{TL}(\textbf{u}_h^n,\textbf{u}_h^{n+1},\textbf{v}_h)
+ a_{NL}^F(\mu^n ;\textbf{u}^{n+1},\textbf{v}_h)-a_{NL}^F(\mu_h^n ;\textbf{u}_h^{n+1},\textbf{v}_h)\}dt
\end{split}
\end{equation}
Deriving term wise estimation as we have done in the previous part and substituting the result in (51) we have
\begin{equation}
\|I_hp-p_h\|_{\bar{\textbf{Q}}}^2 \leq C(T,\textbf{u},p,c) (h^2+dt^{2})
\end{equation}
Now combining the results obtained in the first and second part we have finally arrived at the auxiliary error estimate as follows
\begin{equation}
\|E^A_{\textbf{u}}\|^2_{\bar{\textbf{P}}} + \|E^A_p\|_{\bar{\textbf{Q}}}^2  + \|E^A_{c}\|^2_{\bar{\textbf{P}}} \leq C(T,\textbf{u},p,c) (h^2+ dt^{2})
\end{equation}
This completes the proof.
\end{proof}
\begin{theorem}(Apriori error estimate)
Assuming the same condition as in the previous theorem, 
\begin{equation}
\|\textbf{u}-\textbf{u}_h\|_{\bar{\textbf{P}}}^2+\|p-p_h\|_{\bar{\textbf{Q}}}^2 + \|c-c_h\|^2_{\bar{\textbf{P}}} \leq C' (h^2+ dt^{2})
\end{equation}
where $C'$ depends on T, $\textbf{u}$,p,c.
\end{theorem}
\begin{proof}
By applying triangle inequality, the interpolation inequalities and the result of the previous theorem we will have,
\begin{equation}
\begin{split}
& \|\textbf{u}-\textbf{u}_h\|_{\bar{\textbf{P}}}^2+\|p-p_h\|_{\bar{\textbf{Q}}}^2 + \|c-c_h\|^2_{\bar{\textbf{P}}} \\
& \leq \bar{C} (\|E^I_{\textbf{u}}\|_{\bar{\textbf{P}}}^2 +\|E^I_{p}\|_{\bar{\textbf{Q}}}^2 +  \|E^I_{c}\|_{\bar{\textbf{P}}}^2+ 
 \|E^A_{\textbf{u}}\|_{\bar{\textbf{P}}}^2 +\|E^A_{p}\|_{\bar{\textbf{Q}}}^2 +  \|E^A_{c}\|_{\bar{\textbf{P}}}^2\\
& \leq C'(T,\textbf{u},p,c)(h^2+ dt^{2})
\end{split}
\end{equation}
This completes apriori error estimation.
\end{proof} 

\subsection{Aposteriori error estimation}
\begin{theorem} 
 Assuming the viscosity,diffusion, density and reaction coefficients satisfying the assumptions \textbf{(i)}-\textbf{(iii)} and considering adequately small time step $dt(>0)$, then for sufficiently regular continuous solutions $(\textbf{u},p,c)$ satisfying the assumptions \textbf{(iv)}-\textbf{(v)} and the computed solutions $(\textbf{u}_h,p_h,c_h) \in$ $V_h \times V_h \times Q_h \times V_h$ satisfying (19),  there exists a constant $\tilde{C}(\textbf{R}_h)$, depending upon the computed solution such that
\begin{equation}
 \|\textbf{u}-\textbf{u}_h\|_{\bar{\textbf{P}}}^2+\|p-p_h\|_{\bar{\textbf{Q}}}^2 + \|c-c_h\|^2_{\bar{\textbf{P}}} \leq \bar{C}(\textbf{R}_h) (h^2+ dt^{2})
\end{equation}
where $\textbf{R}_h$ is the residual, defined in (15)
\end{theorem}
\begin{proof}
Estimation of $aposteriori$ error proceeds in the same way as we have done during $auxiliary$ error estimation:
first part contains derivation of $velocity$ and $concentration$ error bound whereas the later part shows estimation of  $pressure$ error. \\
\textbf{First part:} $\forall \hspace{1mm} \textbf{V}$ $\in \bar{\textbf{V}}$
\begin{equation}
\begin{split}
\mu_l \mid \textbf{v} \mid_1^2 + D_{l} \mid d \mid_1^2 + \alpha \|d\|^2 +2 \int_{\Omega} \mu \frac{\partial v_1}{\partial y} \frac{\partial v_2}{\partial x} \leq B(\textbf{u},\mu ;\textbf{V},\textbf{V}) 
\end{split}
\end{equation}
Since the error $\textbf{e}_{\textbf{U}}$ belongs to $ \bar{\textbf{V}}$, replacing the test function $\textbf{V}$ by $\textbf{e}_{\textbf{U}}$ in particular and adding the time derivative terms to the both sides of (59) we have modified (59) as follows:
\begin{multline}
\underbrace{\rho (\frac{\partial e_{\textbf{u}}}{\partial t},e_{\textbf{u}})+ (\frac{\partial e_{c}}{\partial t},e_{c})+\mu_l \| e_{\textbf{u}} \|_1^2+  D_{l} \mid e_c\mid_1^2 + \alpha \|e_c\|^2 }_\textit{LHS}\\
 \leq \underbrace{ ( M \frac{\partial \textbf{e}_{\textbf{U}}}{\partial t}, \textbf{e}_{\textbf{U}}) + B(\textbf{u},\eta(c,\textbf{u});\textbf{e}_{\textbf{U}},\textbf{e}_{\textbf{U}}) +
 2 \int_{\Omega} \mid \mu  \frac{\partial e_{u1}}{\partial y} \frac{\partial e_{u2}}{\partial x} \mid}_\textit{RHS}
\end{multline}
Let us give here a very brief  idea about the proof: we first find lower bound of $\textit{LHS}$ followed by deriving term wise upper bound  for $\textit{RHS}$  and finally combine them to arrive at the desired estimate. Applying $backward$ $Euler$ time discretization rule and second result given by (29) on the first two expressions of the $\textit{LHS}$ involving the time derivative terms we have (60) as follows:
\begin{multline}
 \frac{\rho}{2 dt}(\|e_{\textbf{u}}^{n+1}\|^2-\|e_{\textbf{u}}^n\|^2)+\frac{1}{2 dt}(\|e_{c}^{n+1}\|^2-\|e_{c}^n\|^2)+\mu_l \| e_{\textbf{u}}^{n+1} \|_1^2+  D_{l} \mid e_c^{n+1}\mid_1^2 + \alpha \|e_c^{n+1}\|^2\\
  \leq \textit{LHS} \leq \textit{RHS}
\end{multline}
In order to find upper bound for $\textit{RHS}$ we divide it into three broad parts through splitting the errors in the second argument in the following way:
\begin{equation}
\begin{split}
RHS &= [ ( M \frac{\partial \textbf{e}^n_{\textbf{U}}}{\partial t}, E^{I,n+1}_{\textbf{U}})+ B(\textbf{u}^n,\eta^n; \textbf{U}^{n+1}, E^{I,n+1}_{\textbf{U}}) - B(\textbf{u}^n_h,\eta^n_h; \textbf{U}^{n+1}_h, E^{I,n+1}_{\textbf{U}})] + \\
& \quad   [( M \frac{\partial \textbf{e}^n_{\textbf{U}}}{\partial t}, E^{I,n+1}_{\textbf{U}})+  B(\textbf{u}^n,\eta^n; \textbf{U}^{n+1}, E^{A,n+1}_{\textbf{U}}) - B(\textbf{u}^n_h,\eta^n_h; \textbf{U}^{n+1}_h, E^{A,n+1}_{\textbf{U}})] + \\
& \quad  [ 2 \int_{\Omega} \mid \eta^n   \frac{\partial e_{u1}^{n+1}}{\partial y} \frac{\partial e_{u2}^{n+1}}{\partial x} \mid -a_{TL}^F(e_{\textbf{u}}^{n},\textbf{u}_h^{n+1}, e_{\textbf{u}}^{n+1}) + a_{NL}^F (\mu^n-\mu_h^n; \textbf{u}^{n+1}_h, e_{\textbf{u}}^{n+1})\\
& \quad -a_{TL}^T(e_{\textbf{u}}^n, c_h^{n+1},e_c^{n+1})]\\
& = RHS(I)+ RHS(A)+RHS(R)
\end{split}
\end{equation}
According to the formation of the three parts, the alphabets $I,A,R$ stand for $Interpolation$, $Auxiliary$ and $Remaining$ respectively. Now we are aiming to bring residual into context and for that purpose on integrating the terms in $B(\cdot,\cdot; \cdot, \cdot)$  we have the following relation $\forall$ $\textbf{V} \in \bar{\textbf{V}}$
\begin{multline}
 (M \frac{\textbf{e}_{\textbf{u}}^{n+1}- \textbf{e}_{\textbf{u}}^{n}}{dt}, \textbf{V})+ B(\textbf{u}^n,\mu^n; \textbf{U}^{n+1}, \textbf{V})- B(\textbf{u}^n_h,\mu^n_h; \textbf{U}^{n+1}_h, \textbf{V}) =  \int_{\Omega} \textbf{R}^{n+1}_h \cdot \textbf{V}
\end{multline}
On substituting $\textbf{V}$ in the above expressions by $E^{I,n+1}_{\textbf{U}}$, applying the Cauchy-Schwarz inequality, the standard interpolation estimate and assumption \textbf{(iv)} on the exact solutions subsequently we have the estimated result for $RHS(I)$ as follows:
\begin{equation}
\begin{split}
RHS(I) & = \int_{\Omega} \textbf{R}_h^{n+1} \cdot E^{I,n+1}_{\textbf{U}}
 \leq h^2 \bar{C}_1 \|\textbf{R}_h^{n+1}\|  \\
\end{split}
\end{equation}
The parameters $\bar{C}_1$ comes after applying assumption  \textbf{(iv)} and the interpolation estimate (27). To estimate $RHS(A)$ let us bring here (30) in the following way: $\forall$ $\textbf{V}_h \in \bar{\textbf{V}}_h$
\begin{multline}
(M \frac{\textbf{e}_{\textbf{u}}^{n+1}- \textbf{e}_{\textbf{u}}^{n}}{dt}, \textbf{V}_h)+ B(\textbf{u}^n,\mu^n; \textbf{U}^{n+1}, \textbf{V}_h)- B(\textbf{u}^n_h,\mu^n_h; \textbf{U}^{n+1}_h, \textbf{V}_h)\\
= \sum_{k=1}^{n_{el}} \{(\tau_k' \textbf{R}_h^{n+1}, \mathcal{L}^*(\textbf{u}_h,\mu_h;\textbf{V}_h))_{\Omega_k}+(\tau_k' \textbf{d},\mathcal{L}^*(\textbf{u}_h,\mu_h;\textbf{V}_h))_{\Omega_k} +  \\
((I-\tau_k^{-1}\tau_k) \textbf{R}_h^{n+1}, \textbf{V}_h)_{\Omega_k} +(\tau_k^{-1}\tau_k \textbf{d}, \textbf{V}_h)_{\Omega_k} \} +(\textbf{TE}^{n+1}, \textbf{V}_h)
\end{multline}
where  the matrix $\textbf{d}$= $\sum_{i=0}^{N}(\frac{1}{dt}M\tau_k')^i(\textbf{F} -M\partial_t \textbf{U}_h - \mathcal{L}(\textbf{u}_h;\textbf{U}_h))=\sum_{i=0}^{N}(\frac{1}{dt}M\tau_k')^i \textbf{R}_h$.
On substituting $\textbf{V}_h$ by $E^{A,n+1}_{\textbf{U}}$ in the above equation we have
\begin{equation}
\begin{split}
RHS(A) & =  \sum_{k=1}^{n_{el}} \{(\tau_k' \textbf{R}_h^{n+1}, \mathcal{L}^*(\textbf{u}_h,\mu_h ;E^{A,n+1}_{\textbf{U}}))_{\Omega_k}+(\tau_k' \textbf{d},\mathcal{L}^*(\textbf{u}_h,\mu_h ;E^{A,n+1}_{\textbf{U}}))_{\Omega_k} +  \\
& \quad  ((I-\tau_k^{-1}\tau_k) \textbf{R}_h^{n+1}, E^{A,n+1}_{\textbf{U}})_{\Omega_k} +(\tau_k^{-1}\tau_k \textbf{d}, E^{A,n+1}_{\textbf{U}})_{\Omega_k} \} +(\textbf{TE}^{n+1},E^{A,n+1}_{\textbf{U}})
\end{split}
\end{equation}
Let us see the residual term here more explicitly.
\[
\textbf{R}_h=
  \begin{bmatrix}
 \textbf{f}-\{ \rho \frac{\partial \textbf{u}_h}{\partial t} + \rho (\textbf{u}_h \cdot \bigtriangledown) \textbf{u}_h + \bigtriangledown p_h- \nabla \cdot 2 \mu_h \textbf{D}(\textbf{u}_h)  \} \\
    -\bigtriangledown \cdot \textbf{u}_h \\
 g-(\frac{\partial c_h}{\partial t} - \bigtriangledown \cdot \tilde{\bigtriangledown} c_h + \textbf{u} \cdot \bigtriangledown c_h + \alpha c_h )
  \end{bmatrix}
   = 
  \begin{bmatrix}
 \textbf{R}_{1h}\\
  R_{2h} \\
  R_{3h}
  \end{bmatrix}
\]
where $\textbf{f}=[f_1,f_2]^T$.
The estimation of the terms in $RHS(A)$ follows the same way as we have done in the derivation of auxiliary $apriori$ error estimate. Here we use the observation made upon the bounded property of the elements belonging to finite element spaces $V_h$ and $Q_h$ over each element sub-domain $\Omega_k$ for $k=1,2,...,n_{el}$. On expanding each of the terms, applying that observation on the auxiliary error parts and assumption \textbf{(i)} on the viscosity expression, we have the estimations as follows:
\begin{equation}
\begin{split}
 RHS(A) & \leq  \mid \tau_1 \mid  \sum_{n=0}^{N-1} \sum_{i=1}^4 \tilde{C}_1^i \|\textbf{R}_{1h}^{n+1} \| dt +  h^2 \frac{ \mid \tau_2 \mid}{ \epsilon_1' } \tilde{C}_1^2+ \mid \tau_3 \mid  \sum_{n=0}^{N-1} \sum_{i=1}^4 \tilde{C}_3^i \|R_{3h}^{n+1} \| \\
 & \quad + \mid (\textbf{TE}^{n+1},E^{A,n+1}_{\textbf{U}}) \mid
\end{split}
\end{equation}
where the parameters $\tilde{C}_1^i, \tilde{C}_3^i$ for $i=1,...,4$ contain bounds of auxiliary error part of each variable over each sub-domain. Now the terms involving truncation error has to be estimated in slightly different way as we have done in the previous section. Applying the Cauchy-Schwarz and Young's inequality we have the estimation in the following way:
\begin{equation}
\begin{split}
\mid (\textbf{TE}^{n+1},E^{A,n+1}_{\textbf{U}}) \mid & = \mid (\textbf{TE}^{n+1}, \textbf{e}^{n+1}) \mid + \mid (\textbf{TE}^{n+1}, E^{I,n+1}_\textbf{U}) \mid \\
& \leq \frac{1}{\epsilon_2'} \|\textbf{TE}^{n+1}\|^2 + \frac{\epsilon_2'}{2} (\|\textbf{e}^{n+1}\|^2 + \|E^{I,n+1}_{\textbf{U}} \|^2) \\
& \leq  \frac{1}{\epsilon_2'} \|\textbf{TE}^{n+1}\|^2 +  \frac{\epsilon_2'}{2} \{ \|\textbf{e}^{n+1}\|^2 + h^4 (\frac{1+\theta}{2} \|\textbf{U}^{n+1}\|_2 + \frac{1-\theta}{2} \|\textbf{U}^n\|_2)^2 \} \\
& \leq \frac{1}{\epsilon_2'} \|\textbf{TE}^{n+1}\|^2 +  \frac{\epsilon_2'}{2} \|\textbf{e}^{n+1}\|^2_1 +  h^4 \frac{\epsilon_2'}{2} \bar{C}_4
\end{split}
\end{equation}
In the next step we have carried out the estimation of the remaining terms $RHS(R)$ in the same way as we have proceeded in the earlier section during the proof of Theorem 2. Hence skipping the repetition the estimated result is:
\begin{equation}
\begin{split}
RHS(R) & \leq (C_1'+C_2' \mu_u)\|e^{n+1}_{\textbf{u}}\|_1^2
\end{split}
\end{equation}
where $C_i'$ for i=1,2 are parameters dependent upon the computed solutions.
Now we combine the estimated results from (61) to (68) in (60) and multiplying both sides by $2 dt$ we sum up  for the time steps for $n=0,...,(N-1)$ to reach at the following:
\begin{multline}
\rho \sum_{n=0}^{N-1} ( \|e^{n+1}_{\textbf{u}}\|^2-\|e^{n}_{\textbf{u}}\|^2)+\sum_{n=0}^{N-1} ( \|e^{n+1}_c\|^2-\|e^{n}_c\|^2)+ (2 \eta_l-C_1'-C_2'\eta_s-\epsilon_1' \mid \tau_2 \mid- \epsilon_2')\\
 \sum_{n=0}^{N-1} \|e^{n+1}_{\textbf{u}}\|_1^2 dt+ (2D_l-\epsilon_2') \sum_{n=0}^{N-1} \mid e^{n+1}_{c}\mid_1^2 dt+ (2 \alpha-\epsilon_2') \sum_{n=0}^{N-1} \|e^{n+1}_{c}\|^2 dt \\
 \leq h^2 \sum_{n=0}^{N-1}( \bar{C}_1 \|\textbf{R}_{1h}^{n+1}\| + \frac{ \mid \tau_2 \mid}{ \epsilon_1' } \tilde{C}_1^2+  h^2 \frac{\epsilon_2'}{2} \bar{C}_4) dt + \mid \tau_1 \mid  \sum_{n=0}^{N-1} \sum_{i=1}^4 \tilde{C}_1^i \|\textbf{R}_{1h}^{n+1} \| dt +\\
 \quad   \mid \tau_3 \mid  \sum_{n=0}^{N-1} \sum_{i=1}^4 \tilde{C}_3^i \|R_{3h}^{n+1} \| dt +\frac{1}{\epsilon_2'}  \sum_{n=0}^{N-1} \|\textbf{TE}^{n+1}\|^2 dt \hspace{20 mm}
\end{multline}
We now choose the arbitrary parameters in such a way that all the coefficients in the left hand side can be made positive.
Then taking minimum over the coefficients in the left hand side let us divide both sides by them. Using property (11) associated with the implicit time discretisation scheme and the fact that $\tau_1, \tau_3$ are of order $h^2$, we have arrived at the following relation:
\begin{equation}
\boxed{\|\textbf{u}-\textbf{u}_h\|_{\bar{\textbf{P}}}^2  + \|c-c_h\|_{\bar{\textbf{P}}}^2 \leq C'(\textbf{R}_h) (h^2+dt^{2})}
\end{equation}
This only completes one part of $aposteriori$ estimation and in the next part we combine the corresponding pressure part.\vspace{2mm}\\
\textbf{Second part:} Using (50) we can rewrite (49) in the following form
\begin{equation}
\begin{split}
b(\textbf{v}_h, I^h_p p-p_h) & = (\frac{\partial e_{\textbf{u}}}{\partial t}, \textbf{v}_h) + a^F_{TL}(\textbf{u},\textbf{u},\textbf{v}_h)
-a^F_{TL}(\textbf{u}_h,\textbf{u}_h,\textbf{v}_h) + a_{NL}^F(\mu ;\textbf{u},\textbf{v}_h) \\
& \quad - a_{NL}^F(\mu_h ;\textbf{u}_h,\textbf{v}_h)
\end{split}
\end{equation}
Now we discretize with respect to time and apply the $Cauchy-Schwarz$ inequality, the $Young$s inequality, property \textbf{(b)} of the $trilinear$ form $a_{TL}^F(\cdot, \cdot, \cdot)$ and the above result (70) subsequently during the process of derivation we finally have,
\begin{equation}
\sum_{n=0}^{N-1}  b(\textbf{v}_h, E_p^{A,n+1}) dt \leq \bar{C}'(\textbf{R}_h)(h^2+ dt^{2}) \|\textbf{v}_h\|_1
\end{equation}
Using this estimate equation (51) becomes
\begin{equation}
\|I_hp-p_h\|^2_{L^2(L^2)} \leq \bar{C}''(\textbf{R}_h)(h^2+ dt^{2})
\end{equation} 
Now combining the results obtained in the first and second part and applying the interpolation estimate (27) on pressure interpolation term $E^I_p$, we finally arrive at the following
\begin{equation}
\boxed{\|\textbf{u}-\textbf{u}_h\|^2_{\bar{\textbf{P}}} + \|p-p_h\|_{\bar{\textbf{Q}}}^2  + \|c-c_h\|^2_{\bar{\textbf{P}}} \leq \bar{C}(\textbf{R}_h) (h^2+ dt^{2})}
\end{equation} 
Now this finally completes derivation of $aposteriori$ error estimation. 
\end{proof}
\begin{remark}
These estimations clearly imply that the scheme is $first$ order convergent in space with respect to total norm, whereas in time it is $first$ order convergent for backward Euler time discretization scheme.
\end{remark}

\begin{figure}[htbp]
    \centering 
\begin{subfigure}{0.33\textwidth}
  \includegraphics[width=\linewidth]{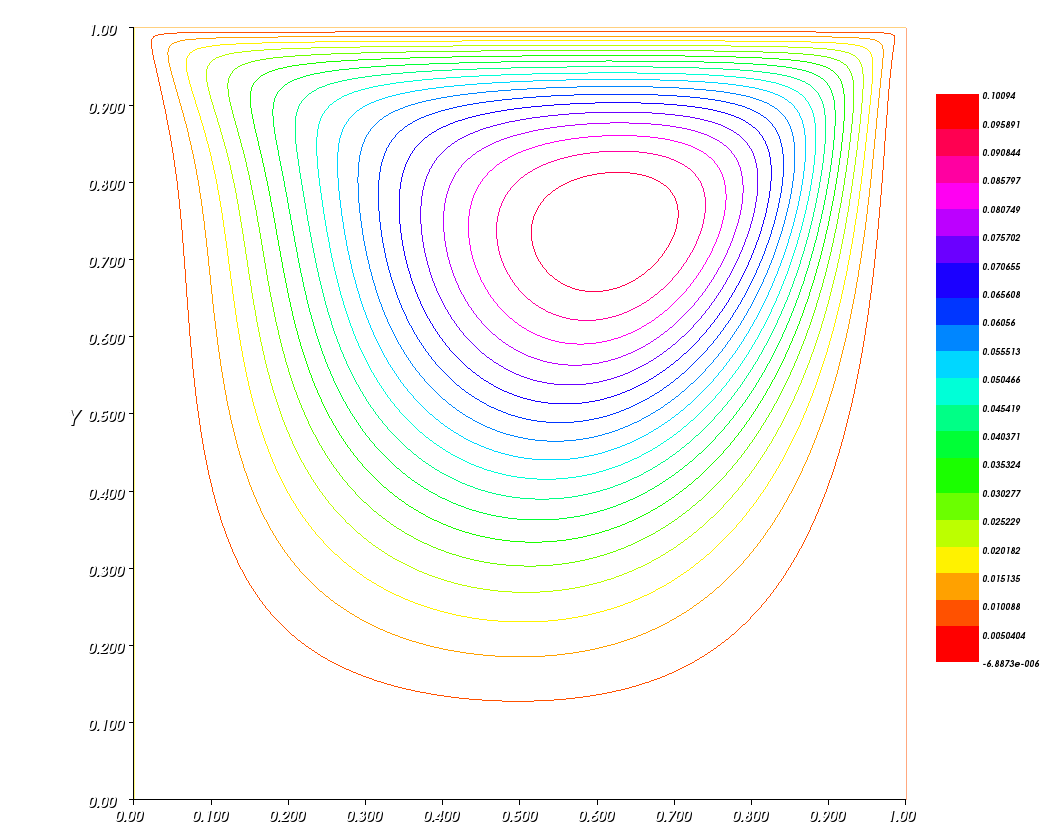}
  \caption{Re=100}
  \label{fig:1}
\end{subfigure}\hfil 
\begin{subfigure}{0.33\textwidth}
  \includegraphics[width=\linewidth]{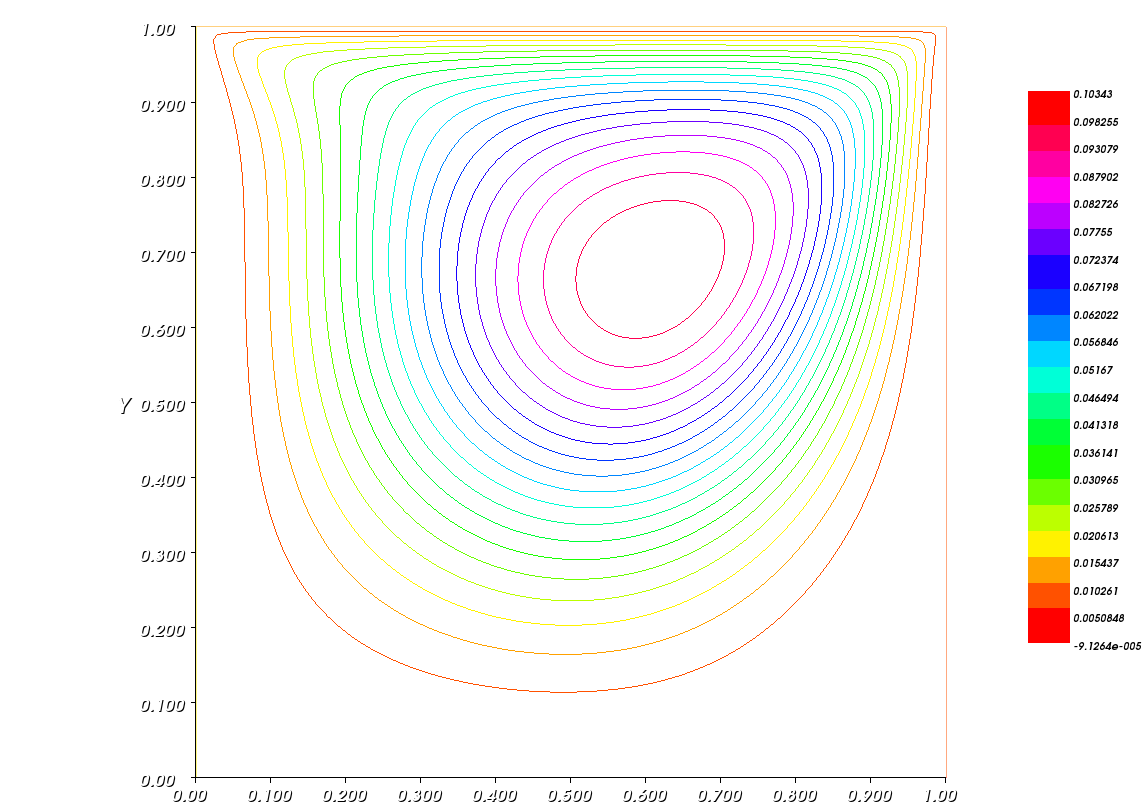}
  \caption{Re=200}
  \label{fig:2}
\end{subfigure}\hfil 
\begin{subfigure}{0.33\textwidth}
  \includegraphics[width=\linewidth]{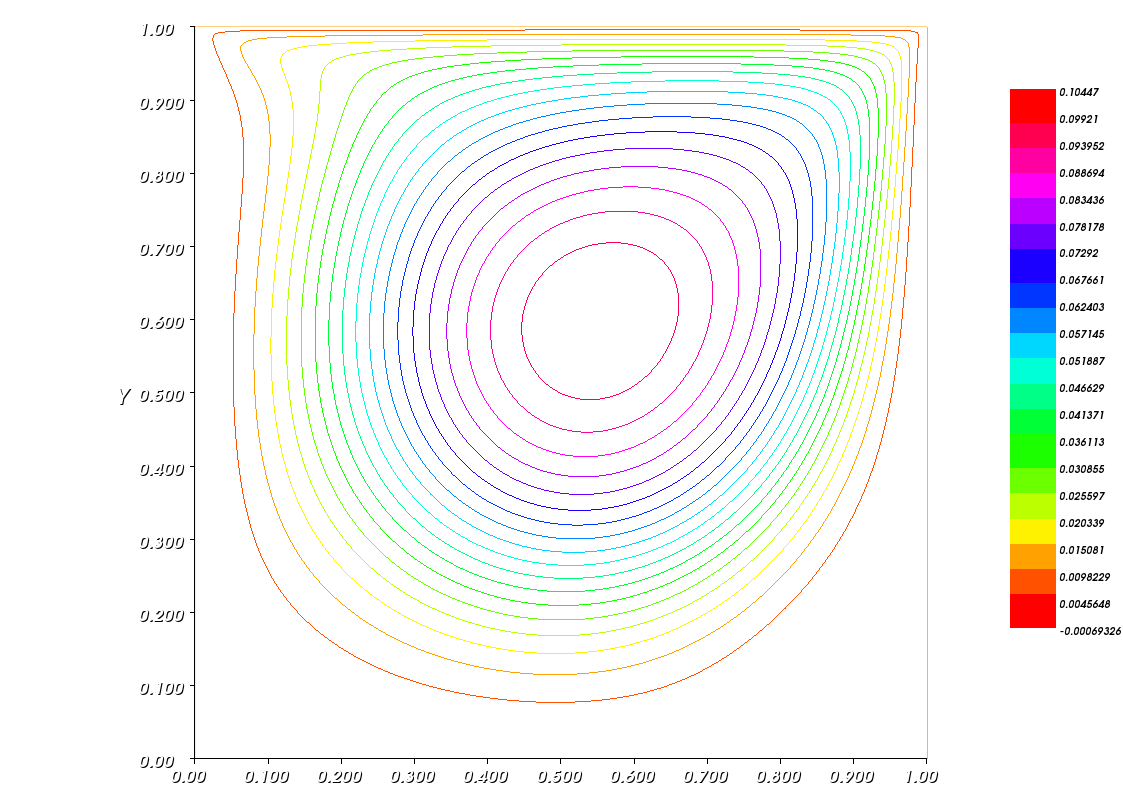}
  \caption{Re=500}
  \label{fig:3}
\end{subfigure}

\medskip
\begin{subfigure}{0.33\textwidth}
  \includegraphics[width=\linewidth]{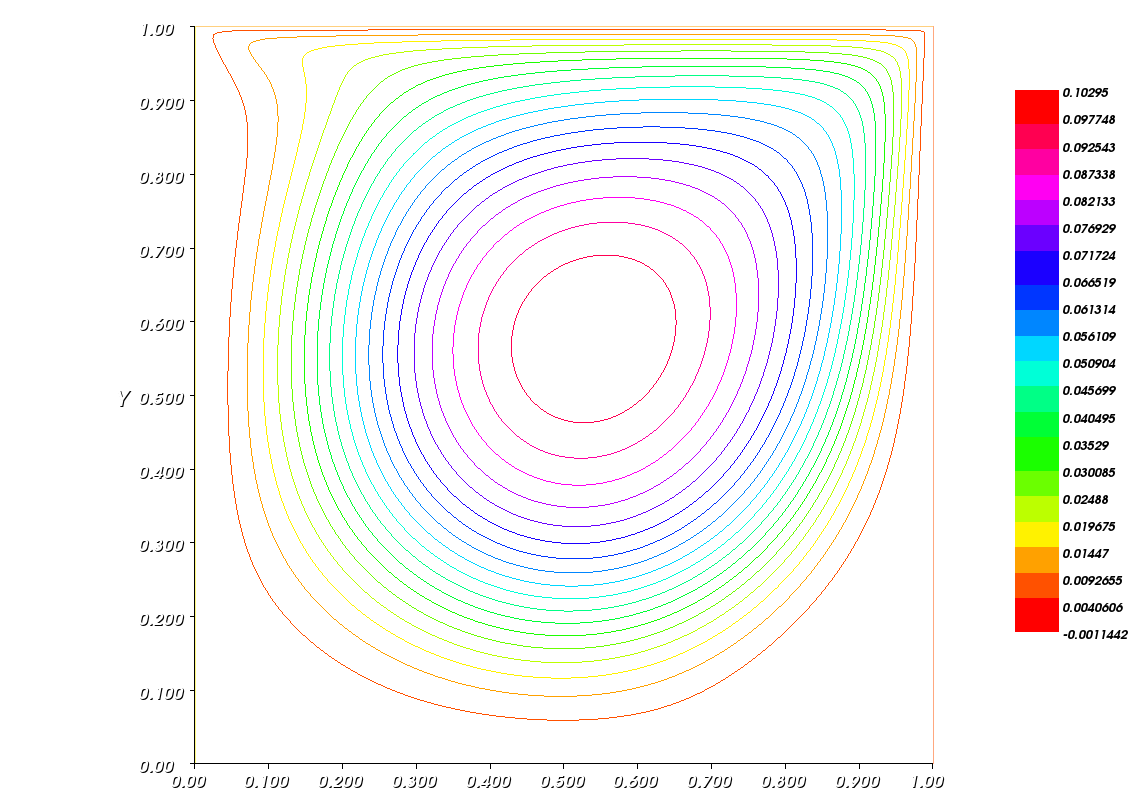}
  \caption{Re=800}
  \label{fig:4}
\end{subfigure}\hfil 
\begin{subfigure}{0.33\textwidth}
  \includegraphics[width=\linewidth]{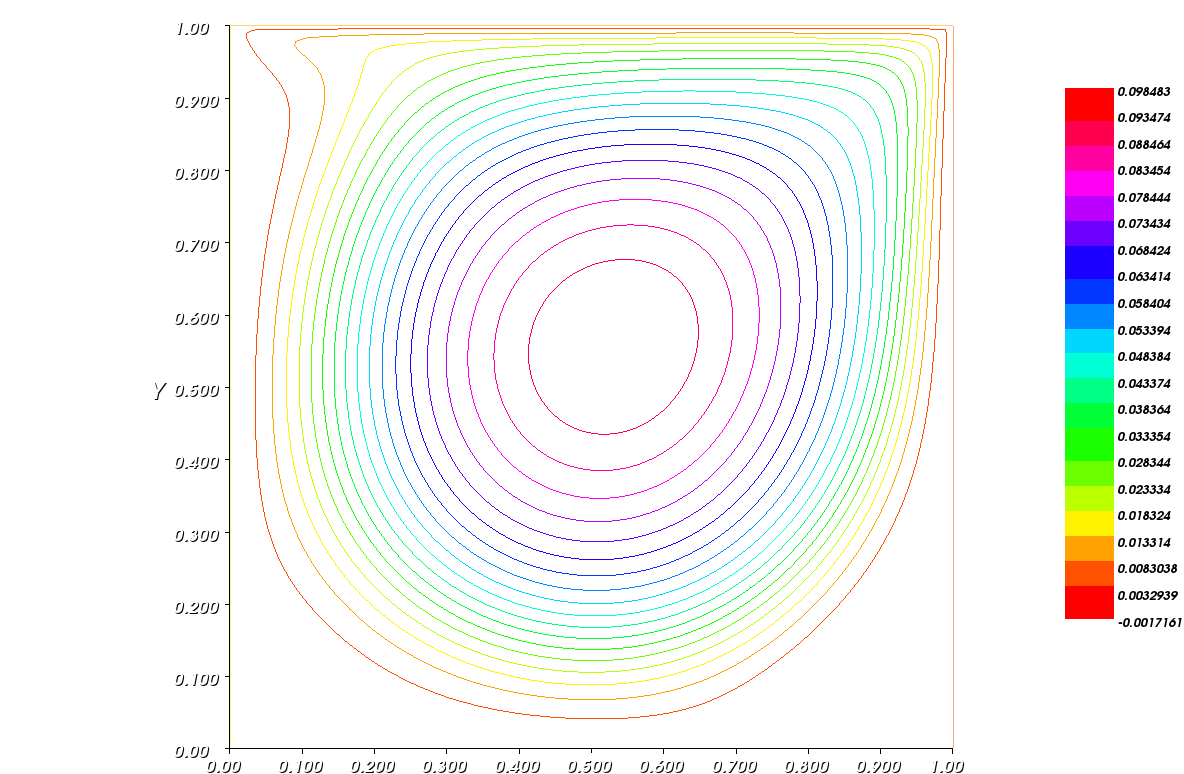}
  \caption{Re=1500}
  \label{fig:5}
\end{subfigure}\hfil 
\begin{subfigure}{0.33\textwidth}
  \includegraphics[width=\linewidth]{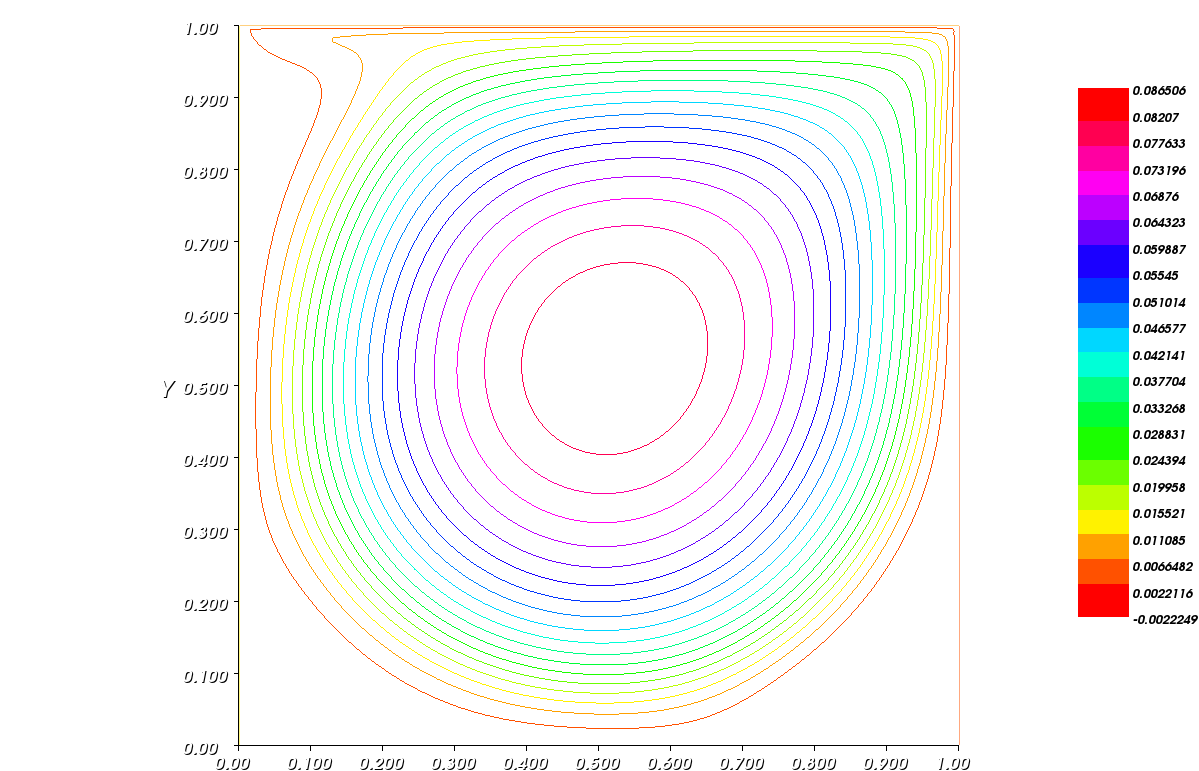}
  \caption{Re=4000}
  \label{fig:6}
\end{subfigure}
\caption{Single lid driven cavity flow streamline results for present method}
\label{fig:images}
\end{figure}

\begin{figure}[htbp]
    \centering 
\begin{subfigure}{0.49\textwidth}
  \includegraphics[width=6.5cm, height=5.5cm]{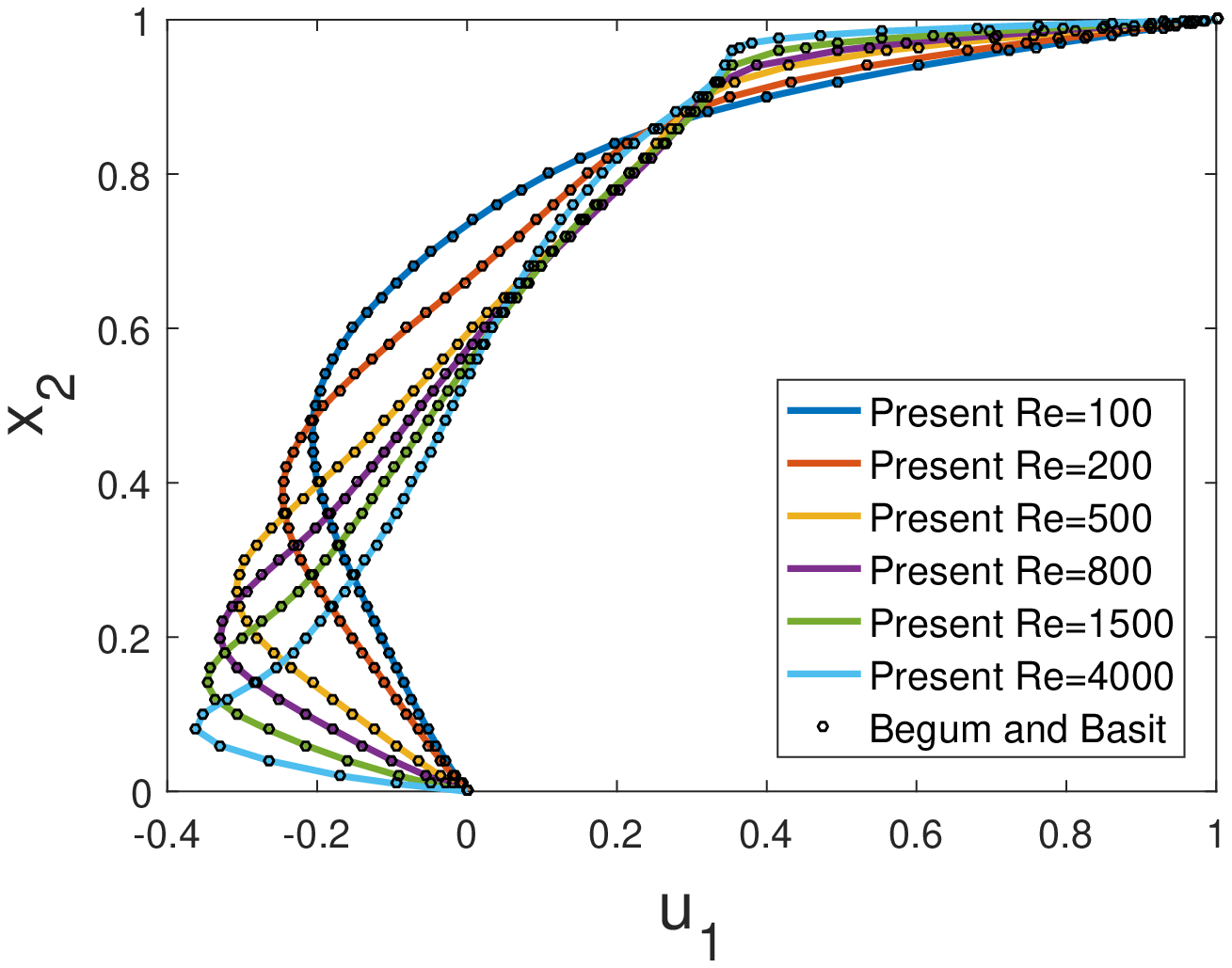}
  \label{fig:1}
\end{subfigure}\hfil 
\begin{subfigure}{0.49\textwidth}
  \includegraphics[width=6.5cm, height=5.5cm]{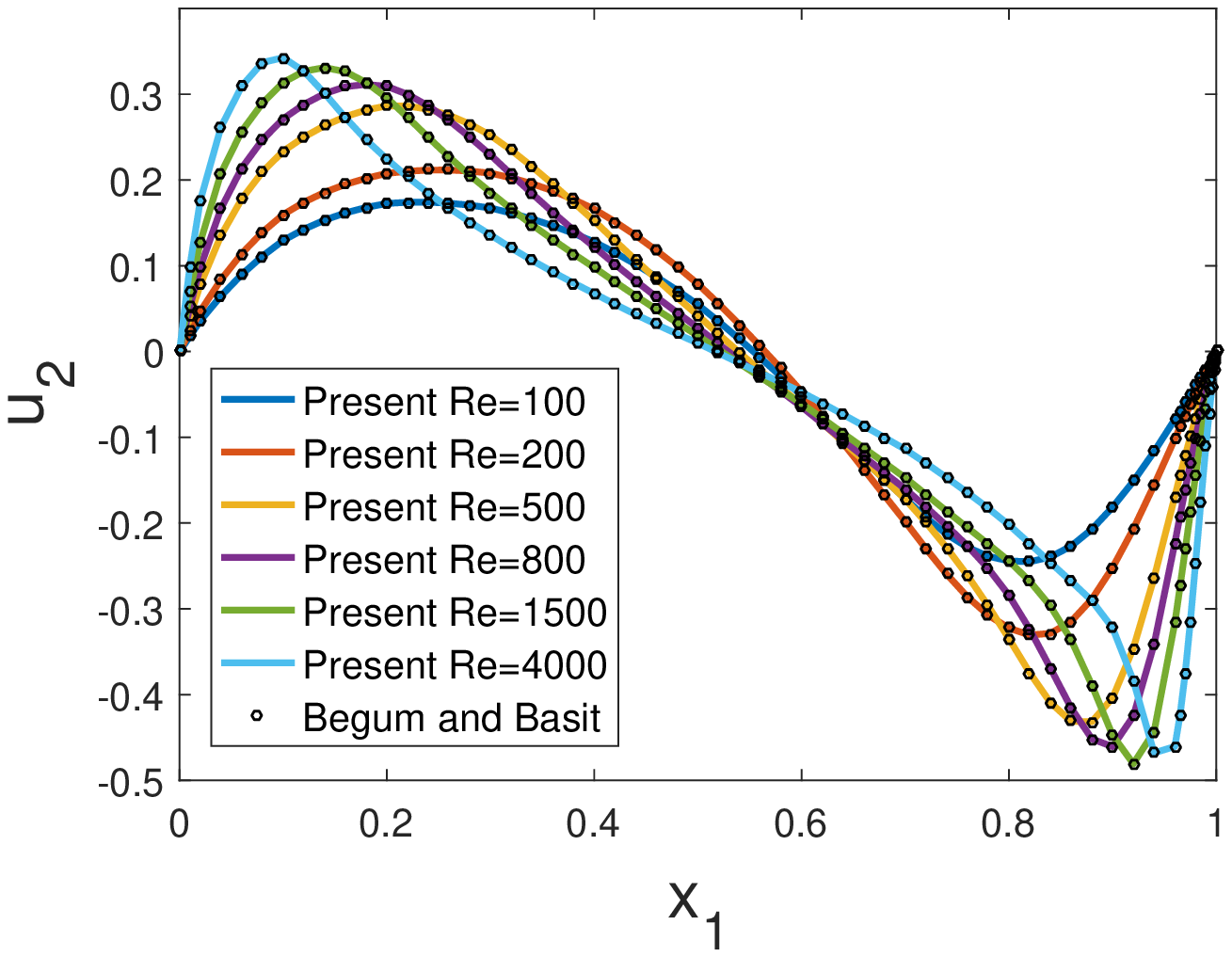}
  \label{fig:2}
\end{subfigure}\hfil 
\caption{Comparison of concentration error $(e_{c})$ and pressure error ($e_p$) plots respectively }
\label{fig:images}
\end{figure}

\section{Numerical experiments}
\subsection{Validation: Lid driven cavity problem}
The preformance of the present method is validated through this driven cavity benchmark problem, where the results are compared with the studies of Begum and Basit in \cite{13}. We have tested for single lid driven square (0,1) $\times$ (0,1) cavity problem with unit horizontal velocity at upper most wall and zero at the remaining boundary walls. As well as vanishing vertical velocities are considered at all the walls of the boundary. The experiments are performed at 256 $\times$ 256 grids for the Reynolds numbers, $Re$ =100, 200, 500, 800, 1500, 4000.
\subsubsection*{Results and discussion}
Figure 1 presents the driven cavity streamline plots generated for $Re$= 100, 200, 500, 800, 1500, 4000 respectively. A standard downward shifting pattern of the primary vortex can be seen in the figures for increasing $Re$ values. Besides figure 2 exhibits the graphical comparison of the horizontal and vertical velocity plots with \cite{13} for wide range of $Re$. All these results indicate equally well performance of the present method and validates its numerical accuracy.

\subsection{Convergence experiments}
This subsection is dedicated in verifying the order of convergence results experimentally for different test cases which include experiments with different $Re$ values, variable viscosity and diffusion coefficients. These experiments are carried out on a bounded square domain $\Omega$ (0,1) $\times$ (0,1) with homogeneous Dirichlet boundary condition. Piecewise linear functions are considered for approximating all three variables. The exact solutions for all the experiments are: \\
$\textbf{u}=(e^{-t}x^2(x-1)^2y(y-1)(2y-1),-e^{-t}x(x-1)(2x-1)y^2(y-1)^2)$ \vspace{1mm}\\
$p= e^{-t}(3x^2+3y^2-2)$ and $c=e^{-t}xy(x-1)(y-1)$\vspace{1mm}\\
Here we mention the error notations corresponding to each variable in the following:\\
Let $e_{\textbf{u}}$= $\sqrt{\|\textbf{u}-\textbf{u}_h\|_{\bar{\textbf{P}}}^2}$, $e_{p}$= $\sqrt{\|p-p_h\|_{\bar{\textbf{Q}}}^2}$ and $e_{c}$= $\sqrt{\|c-c_h\|_{\bar{\textbf{P}}}^2}$ where the norms are already defined in section 3 and the term $total$ specifies the additive result of all these errors.

\begin{table}[]
\centering
\begin{tabular}{|p{7mm}|p{3mm}|p{3mm}|p{10mm}|p{7mm}|p{10mm}|p{7mm}|p{10mm}|p{7mm}|p{10mm}|p{7mm}|}
    \hline
Re & $dt$ & $\frac{1}{h}$ & $e_{\textbf{u}}$ & RoC & $e_c$& RoC & $e_p$ & RoC &Total & RoC\\ [1mm]
 \hline
& $\frac{1}{10}$& 10 & 4.48$e^{-3}$ &  & 8.53$e^{-3}$ & & 1.58$e^{-1}$& &1.59$e^{-1}$ & \\[1mm]
100 & $ \frac{1}{20}$&   20 & 1.86$e^{-3}$ & 1.268  & 4.75$e^{-3}$ & 0.843 & 8.33$e^{-2}$ & 0.928 & 8.34$e^{-2}$  &0.927\\  [1mm]
& $  \frac{1}{40}$&   40 & 8.68$e^{-4}$ & 1.098  & 2.53$e^{-3}$ & 0.908 &  4.31$e^{-2}$ & 0.951& 4.31$e^{-2}$  &0.951 \\   [1mm]
& $ \frac{1}{80}$ &  80 & 3.84$e^{-4}$ & 1.177  & 1.31$e^{-3}$ & 0.956 &2.19$e^{-2}$& 0.973 & 2.20$e^{-2}$ &0.973\\ 
    \hline    
  &  $\frac{1}{10}$& 10 & 5.55$e^{-3}$ &  & 8.53$e^{-3}$ & & 1.58$e^{-1}$& &1.59$e^{-1}$ & \\[1mm]

 & $ \frac{1}{20}$&   20 & 2.70$e^{-3}$ & 1.039  & 4.75$e^{-3}$ & 0.843 & 8.32$e^{-2}$ & 0.927 & 8.34$e^{-2}$  & 0.927\\  [1mm]
500 & $  \frac{1}{40}$&   40 & 1.33$e^{-3}$ & 1.017  & 2.53$e^{-3}$ & 0.908 &  4.30$e^{-2}$ & 0.951 & 4.31$e^{-2}$  & 0.951 \\   [1mm]
& $ \frac{1}{80}$ &  80 & 5.69$e^{-4}$ & 1.162  & 1.30$e^{-3}$ & 0.956 & 2.19$e^{-2}$& 0.973 & 2.20$e^{-2}$ & 0.973 \\ 
\hline
& $\frac{1}{10}$& 10 & 6.48$e^{-3}$ &  & 8.53$e^{-3}$ & & 1.58$e^{-1}$& &1.59$e^{-1}$ & \\[1mm]
 & $ \frac{1}{20}$&   20 & 3.79$e^{-3}$ & 0.643  & 4.75$e^{-3}$ & 0.843 & 8.33$e^{-2}$ & 0.928 & 8.35$e^{-2}$  & 0.927\\  [1mm]
5000 & $  \frac{1}{40}$&   40 & 2.02$e^{-3}$ & 0.819  & 2.53$e^{-3}$ & 0.908 &  4.31$e^{-2}$ & 0.951 & 4.32$e^{-2}$  & 0.951 \\   [1mm]
& $ \frac{1}{80}$ &  80 & 1.02$e^{-3}$ & 0.773  & 1.31$e^{-3}$ & 0.956 & 2.19$e^{-2}$& 0.973 & 2.20$e^{-2}$ & 0.973	 \\ 
    \hline    
    & $\frac{1}{10}$& 10 & 6.48$e^{-3}$ &  & 8.53$e^{-3}$ & & 1.58$e^{-1}$& &1.59$e^{-1}$ & \\[1mm]
 & $ \frac{1}{20}$&   20 & 3.79$e^{-3}$ & 0.643  & 4.75$e^{-3}$ & 0.843 & 8.33$e^{-2}$ & 0.928 & 8.35$e^{-2}$  & 0.927\\  [1mm]
10000 & $  \frac{1}{40}$&   40 & 2.02$e^{-3}$ & 0.819  & 2.53$e^{-3}$ & 0.908 &  4.31$e^{-2}$ & 0.951 & 4.32$e^{-2}$  & 0.951 \\   [1mm]
& $ \frac{1}{80}$ &  80 & 1.02$e^{-3}$ & 0.773  & 1.31$e^{-3}$ & 0.956 & 2.19$e^{-2}$& 0.973 & 2.20$e^{-2}$ & 0.973	 \\ 
    \hline    
    & $\frac{1}{10}$& 10 & 6.48$e^{-3}$ &  & 8.53$e^{-3}$ & & 1.58$e^{-1}$& &1.59$e^{-1}$ & \\[1mm]
 & $ \frac{1}{20}$&   20 & 3.79$e^{-3}$ & 0.643  & 4.75$e^{-3}$ & 0.843 & 8.33$e^{-2}$ & 0.928 & 8.35$e^{-2}$  & 0.927\\  [1mm]
50000 & $  \frac{1}{40}$&   40 & 2.02$e^{-3}$ & 0.819  & 2.53$e^{-3}$ & 0.908 &  4.31$e^{-2}$ & 0.951 & 4.32$e^{-2}$  & 0.951 \\   [1mm]
& $ \frac{1}{80}$ &  80 & 1.02$e^{-3}$ & 0.773  & 1.31$e^{-3}$ & 0.956 & 2.19$e^{-2}$& 0.973 & 2.20$e^{-2}$ & 0.973	 \\ 
    \hline    
\end{tabular}
\caption{  Errors and respective order of convergences for $Re$=50 at $T=1$}
    \end{table}

\begin{table}[]
\centering
\begin{tabular}{|p{7mm}|p{3mm}|p{3mm}|p{10mm}|p{7mm}|p{10mm}|p{7mm}|p{10mm}|p{7mm}|p{10mm}|p{7mm}|}
    \hline
Re & $dt$ & $\frac{1}{h}$ & $e_{\textbf{u}}$ & RoC & $e_c$& RoC & $e_p$ & RoC &Total & RoC\\ [1mm]
 \hline
& $\frac{1}{10}$& 10 & 4.48$e^{-3}$ &  & 8.53$e^{-3}$ & & 1.58$e^{-1}$& &1.59$e^{-1}$ & \\[1mm]
100 & $ \frac{1}{20}$&   20 & 1.86$e^{-3}$ & 1.268  & 4.75$e^{-3}$ & 0.843 & 8.33$e^{-2}$ & 0.928 & 8.34$e^{-2}$  &0.927\\  [1mm]
& $  \frac{1}{40}$&   40 & 8.68$e^{-4}$ & 1.098  & 2.53$e^{-3}$ & 0.908 &  4.31$e^{-2}$ & 0.951& 4.31$e^{-2}$  &0.951 \\   [1mm]
& $ \frac{1}{80}$ &  80 & 3.84$e^{-4}$ & 1.177  & 1.31$e^{-3}$ & 0.956 &2.19$e^{-2}$& 0.973 & 2.20$e^{-2}$ &0.973\\ 
    \hline    
\end{tabular}
\caption{  Errors and respective order of convergences for $Re$=50 at $T=1$}
    \end{table}

\subsubsection{Test case 1: Weak coupling}
This coupling indicates viscosity coefficient to be independent of solute concentration. Based on this conception we have further studied this case in terms of following two sub cases. For both the sub cases we have hypothetically considered constant diffusion and reaction coefficients, $D_1$=$D_2$=0.01, $\alpha$=0.01. 
\subsubsection*{Sub case 1:}
We have assumed constant viscosity coefficient in this sub case. Consequently $Re$ has been occurred to determine the nature of the flow. We have carried out convergence tests for particularly five $Re$ values such as 100, 500, 5000, 10000 and 50000 to provide an overall idea about the performance of the present method for small to large Reynolds numbers. 
\subsubsection*{Sub case 2:}
In this sub case we have performed the experiment for a specific viscosity coefficient of Casson fluid. Often the flow of physiological fluids such as blood into the vessels is represented through non-Newtonian Casson model. Consequently the physical properties such as hematocrit, protein content in plasma, vessel diameter etc inevitably affect the fluid viscosity. Few of these influences are reflected through viscosity coefficient \cite{15}, $\mu(J_2)$ = $(k_0(\bar{c})+k_1(\bar{c})J_2^{\frac{1}{4}}) J_2^{-\frac{1}{2}}$,  where $k_0(\bar{c})$ and $k_1(\bar{c})$ can be determined by fitting the curve for shear stress vs viscosity to some physical viscometric data. Here the following values are taken for these parameters from the study  \cite{15}: $k_0(\bar{c})$= 0.1937 $(Pa)^{\frac{1}{2}}$ and $k_1(\bar{c})$= 0.055 (Pa $ s)^{\frac{1}{2}}$. Let us identify this case as 'variable viscosity' in one way coupling results.
\subsubsection*{Results and discussion}
Tables put up the experimentally computed values of errors $e_{\textbf{u}}$, $e_p$, $e_c$ and $total$ with respect to the specified norms and their respective order of convergences for $Re$= 100, 500, 5000, 10000, 50000 and variable viscosity case. These results clearly show that the present method successfully achieved the theoretically established unit rate of convergence for specified norms even for high convective flows. 

\begin{figure}[]
   \centering 
\begin{subfigure}{0.33\textwidth}
  \includegraphics[width=\linewidth]{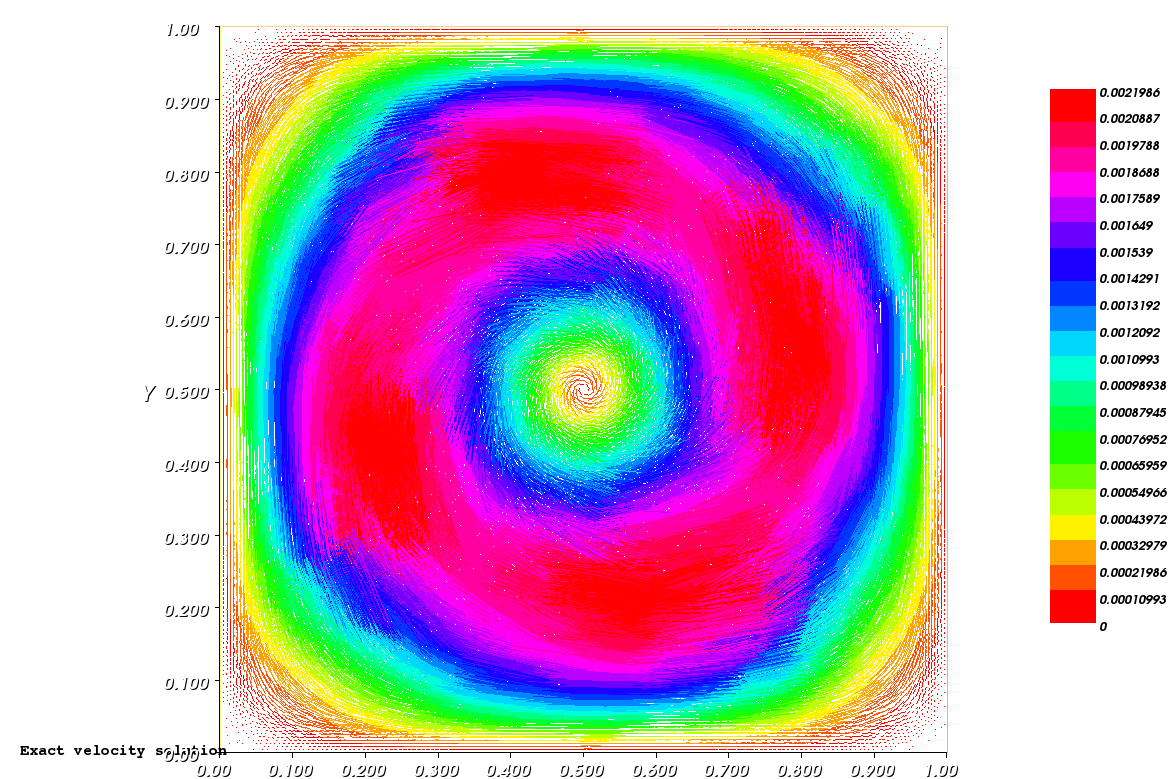}
  \caption{Exact velocity solution}
  \label{fig:1}
\end{subfigure}\hfil 
\begin{subfigure}{0.33\textwidth}
  \includegraphics[width=\linewidth]{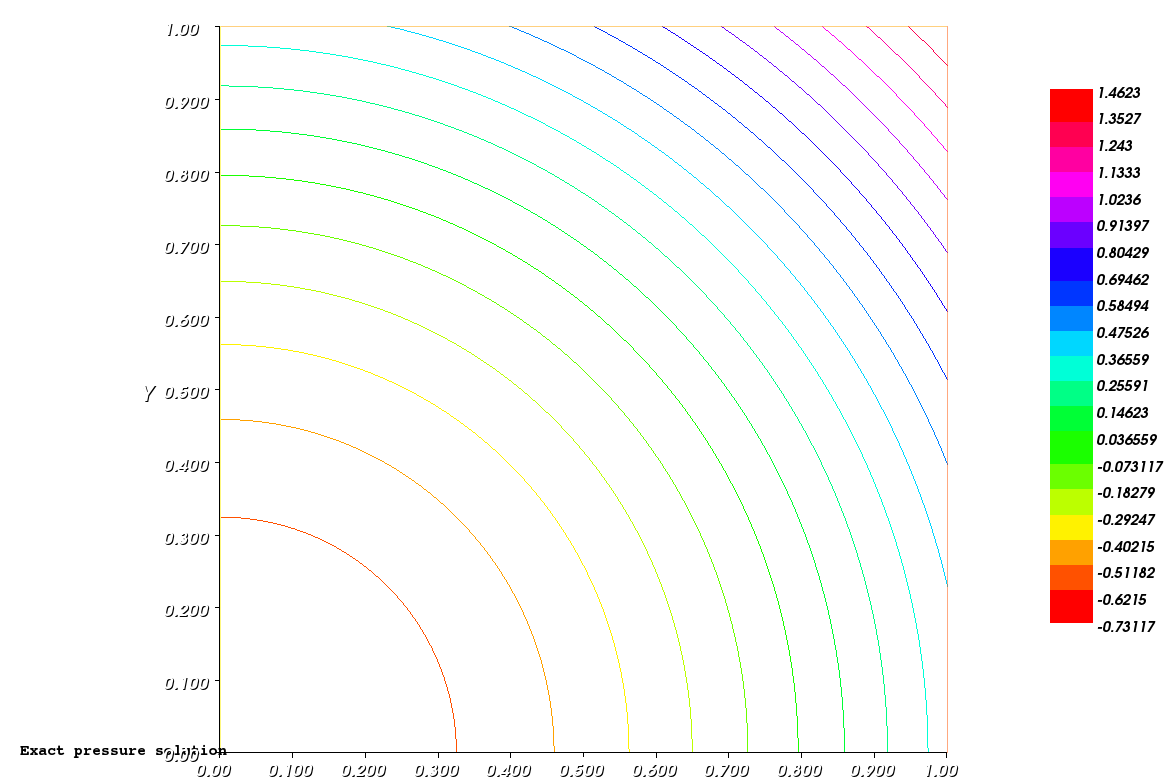}
  \caption{Exact pressure solution}
  \label{fig:2}
\end{subfigure}\hfil 
\begin{subfigure}{0.33\textwidth}
  \includegraphics[width=\linewidth]{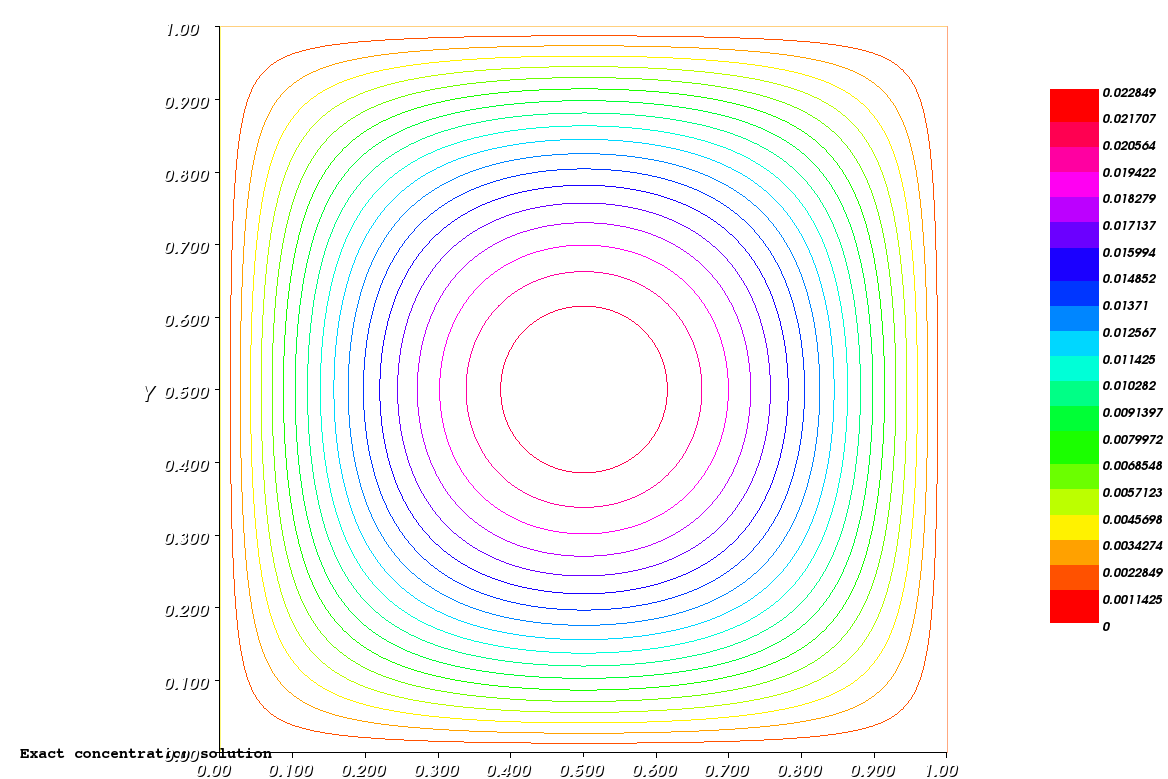}
  \caption{Exact concentration solution}
  \label{fig:3}
\end{subfigure}

\medskip
\begin{subfigure}{0.33\textwidth}
  \includegraphics[width=\linewidth]{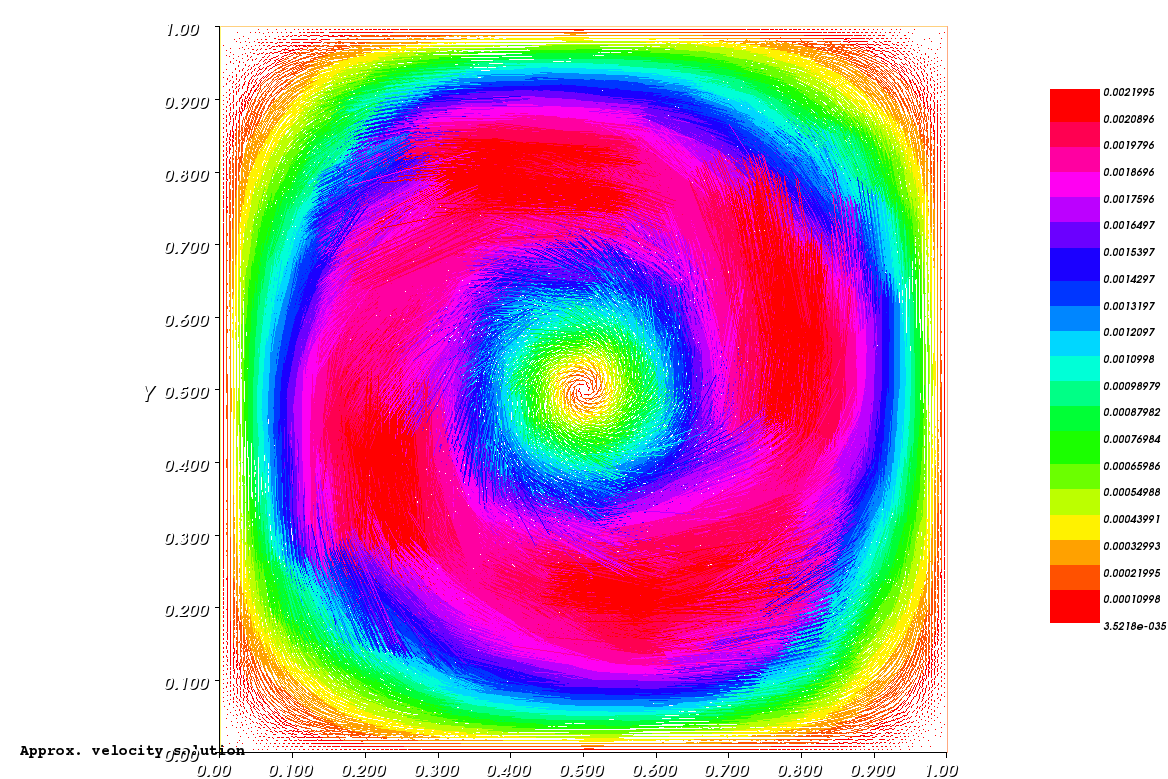}
  \caption{Approximate velocity solution}
  \label{fig:4}
\end{subfigure}\hfil 
\begin{subfigure}{0.33\textwidth}
  \includegraphics[width=\linewidth]{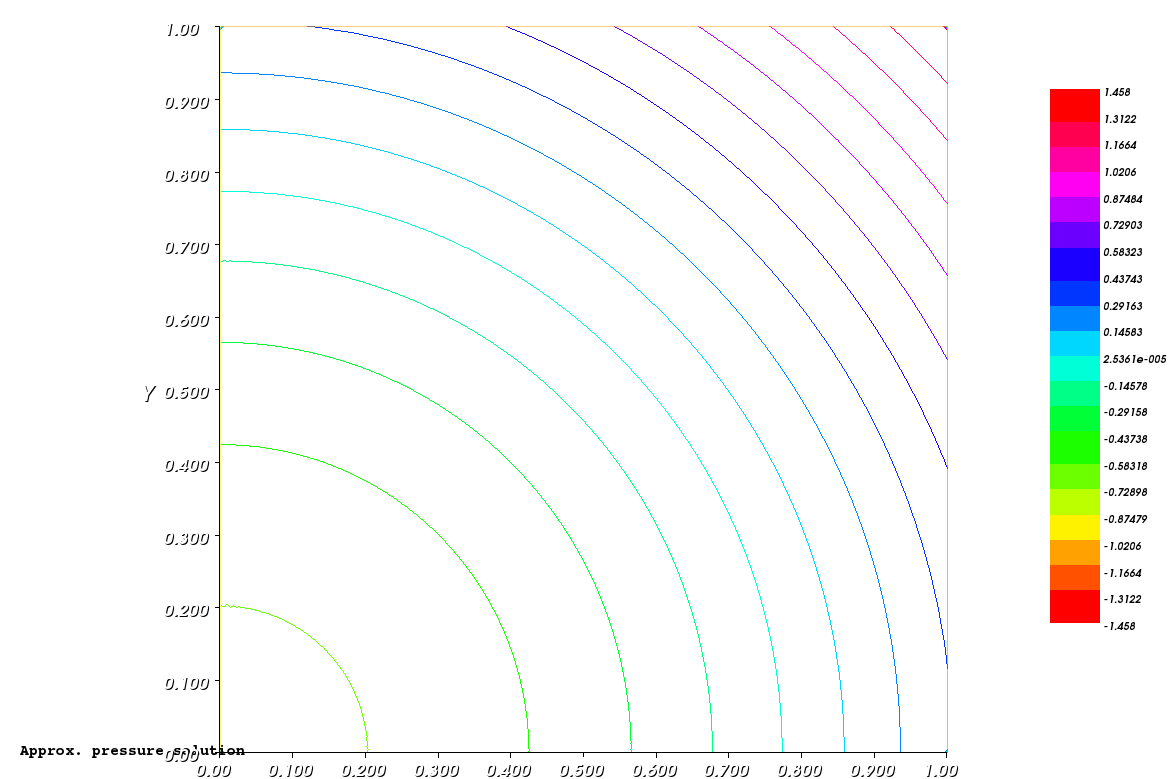}
  \caption{Approximate pressure solution}
  \label{fig:5}
\end{subfigure}\hfil 
\begin{subfigure}{0.33\textwidth}
  \includegraphics[width=\linewidth]{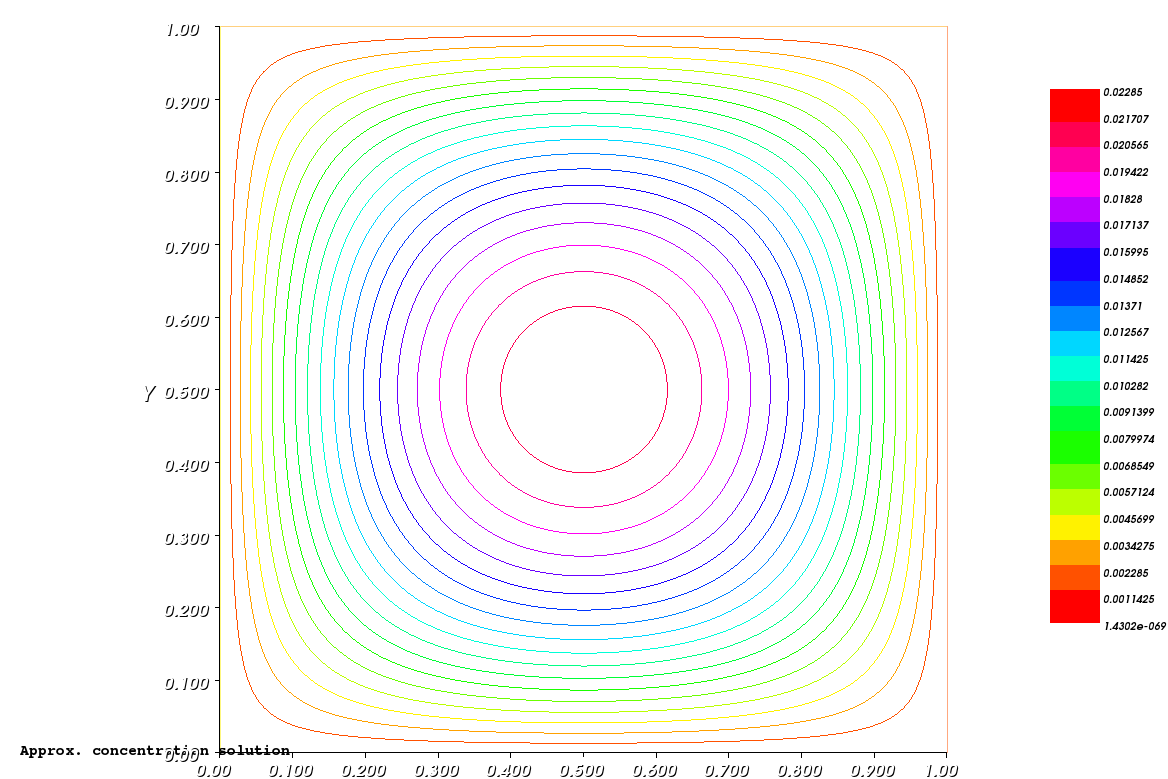}
  \caption{Approximate concentration solution}
  \label{fig:6}
\end{subfigure}
\caption{Comparison between the exact and approximate solution plots at grid size 160 $\times$ 160}
\label{fig:images}
\end{figure}
\subsubsection{Test case 2: Strong coupling}
There are significant research works of practical importance going on in the fields of biomedical engineering, fluid dynamics, chemical engineering with solute concentration dependent viscosity coefficient. Our aim being only to verify the numerical performance of the present method, here we have dealt with two different expression of variable viscosity of some physical relevance. The variable diffusion coefficients are hypothetically chosen as $D_1=e^{-t}y^2(y-1)^2(2y-1)^2x^4(x-1)^4$ and $D_2= e^{-t}x^2(x-1)^2(2x-1)^2y^4(y-1)^4$ in this case. The true solutions remain the same as in the previous experiments.
\subsubsection*{Sub case 1:}
$Antonova$ in \cite{35} has come up with a mathematical expression of blood viscosity representing its rheological behavior during the presence of some other concentrated quantity. For our convergence analysis purpose we have used the expression of the solute concentration dependent viscosity coefficient, $\eta(c)$= $\eta_0 (1+ K c)$, where plasma viscosity $\eta_0$=0.16 $Pa$ $s$, the shape factor of spherical particle $K$=0.25 and $c$ is the concentration of solute. We mark this sub case as 'physiological example' in the strong coupling results.
\subsubsection*{Sub case 2:}
This experiment is motivated by a study on slurry suspension \cite{15}, where $Harrison$ and $Maltman$ has worked with variable viscosity which represents exponential suspension nature of a fluid material during its mixing up into another fluid. The expression of viscosity considered here is $\eta(c)$= $A e^{Bc}$, where the consistency factor $A$=0.129 $Pa$ $s$ and dimensionless constant $B$=0.101. Let us name this case by 'slurry suspension'.
\subsection*{Results and discussion}
The comparison of the exact and approximate solution plots shows that this method approximates all the variables finely for this strongly coupled system too. Table 4 and table 5 clearly indicate that first order convergence have been reached by the method for both the sub cases. The order of convergence plots in figure 6 reassure this fact. All these results establish together the excellent numerical performance of the present method for the coupled system containing concentration dependent viscosity coefficient.

\end{document}